\documentclass[a4paper, 12pt]{amsart}

\usepackage{manfnt}

\usepackage{amssymb}
\usepackage{amsthm}
\usepackage{amsmath}

\usepackage{comment}

\usepackage{url}

\usepackage[mathscr]{eucal}

\usepackage{enumitem}

\theoremstyle{plain}
\newtheorem{thm}{Theorem}[section]
\newtheorem{prop}[thm]{Proposition}
\newtheorem{cor}[thm]{Corollary}
\newtheorem{lem}[thm]{Lemma}

\theoremstyle{definition}
\newtheorem{df}{Definition}[section]

\newtheorem{ques}[thm]{Question}
\newtheorem{conj}[thm]{Conjecture}

\theoremstyle{remark}
\newtheorem{rmk}{Remark}[section]
\newtheorem*{ac}{Acknowledgements}

\newcommand{\zz}{\mathbb{Z}}

\newcommand{\rr}{\mathbb{R}}

\newcommand{\yoemph}[1]{\emph{#1}}

\newcommand{\yoomega}{\omega_{0}}

\newcommand{\supp}{\text{{\rm supp}}}

\newcommand{\yosub}{\subseteq}

\DeclareMathOperator{\yodiam}{diam}

\newcommand{\yocpm}[1]{\mathrm{CPM}(#1)}

\newcommand{\met}[1]{\mathrm{Met}(#1)}
\DeclareMathOperator{\metdis}{\mathcal{D}}

\newcommand{\ult}[2]{\mathrm{UMet}(#1; #2)}
\DeclareMathOperator{\umetdis}{\mathcal{UD}}

\newcommand{\compm}[1]{\mathrm{Comp}(#1)}

\newcommand{\yocase}[2]{Case #1.~[#2]:}

\newcommand{\oumetbd}[1]{\mathrm{BMet}(#1)}

\newcommand{\yomainmap}{E}

\newcommand{\yomainmapf}{\widetilde{\yomainmap}}

\newcommand{\yofmo}{\mathbf{L}_{fin}^{\infty}}

\newcommand{\yofmosp}[1]{\yofmo(#1)}

\newcommand{\yofmodisn}{\mathbf{D}_{L^{\infty}}}

\newcommand{\yoccl}{\mathfrak{C}}
\newcommand{\yocclq}[1]{\yoccl(#1)}

\newcommand{\yoslameto}{\mathbf{M}_{<\infty}}
\newcommand{\yoslamet}[1]{\yoslameto(#1)}

\newcommand{\yotopemb}[2]{\mathrm{I}^{#1}_{#2}}

\newcommand{\yoosmisf}{\mathrm{F}}
\newcommand{\yoosmisfsp}[1]{\yoosmisf\!\left(#1\right)}
\newcommand{\yoosmisfspq}[2]{\yoosmisf^{#2}\left(#1\right)}
\newcommand{\yoosmisfdis}[1]{\yoosmisf\!\left[#1\right]}
\newcommand{\yoosmisfemb}[1]{\yotopemb{\yoosmisf}{#1}}

\newcommand{\yoosmisa}{\$}
\newcommand{\yoosmisasp}[1]{\yoosmisa\!\left(#1\right)}
\newcommand{\yoosmisadis}[1]{\yoosmisa\!\left[#1\right]}
\newcommand{\yoosmisaemb}[1]{\yotopemb{\yoosmisa}{#1}}

\newcommand{\yoosmisc}{\mathrm{L}}
\newcommand{\yoosmiscsp}[1]{\yoosmisc\!\left(#1\right)}
\newcommand{\yoosmiscspw}[1]{\yoosmisc^{0}\!\left(#1\right)}
\newcommand{\yoosmiscspc}[2]{\yoosmisc^{#2}\!\left(#1\right)}
\newcommand{\yoosmiscdis}[1]{\yoosmisc\!\left[#1\right]}
\newcommand{\yoosmiscemb}[1]{\yotopemb{\yoosmisc}{#1}}

\newcommand{\yowas}{\mathrm{W}_{1}}
\newcommand{\yowasspw}[1]{\mathcal{RP}\!\left(#1\right)}
\newcommand{\yowassp}[1]{\yowas\!\left(#1\right)}
\newcommand{\yowasspc}[2]{\yowas^{#2}\!\left(#1\right)}
\newcommand{\yowasdis}[1]{\yowas\!\left[#1\right]}
\newcommand{\yowasemb}[1]{\yotopemb{\yowas}{#1}}

\newcommand{\yoosmisdw}{\yoosmisa\yowas}
\newcommand{\yoosmisdwsp}[1]{\yoosmisdw\!\left(#1\right)}
\newcommand{\yoosmisdwdis}[1]{\yoosmisdw\!\left[#1\right]}

\newcommand{\yoosmislldw}{\yoosmisc\yoosmisa\yowas}
\newcommand{\yoosmislldwsp}[1]{\yoosmislldw\!\left(#1\right)}
\newcommand{\yoosmislldwdis}[1]{\yoosmislldw\!\left[#1\right]}

\newcommand{\yolip}{\mathbf{Lip}_{1}}
\newcommand{\yolipsp}[1]{\yolip\left(#1\right)}

\newcommand{\yocoeff}[1]{\gamma_{#1}}
\newcommand{\yocoefff}[1]{\xi_{#1}}

\newcommand{\yodmaso}{\boldsymbol{\delta}}
\newcommand{\yodmas}[1]{\yodmaso_{#1}}

\newcommand{\yomasq}{\Psi}

\newcommand{\yowdwd}{\mathcal{WD}}
\newcommand{\yowdc}{\mathcal{O}}
\newcommand{\yowdp}{\mathbf{y}}
\newcommand{\yowda}{\mathbf{a}}
\newcommand{\yowdg}{\varphi}
\newcommand{\yowds}{\sigma}
\newcommand{\yowdce}[1]{\mathcal{O}_{#1}}
\newcommand{\yowdpe}[1]{\yowdp_{#1}}
\newcommand{\yowdae}[1]{\yowda_{#1}}
\newcommand{\yowdge}[1]{\yowdg_{#1}}
\newcommand{\yowdse}[1]{\yowds_{#1}}

\newcommand{\yosubmap}{\mathbf{J}}

\newcommand{\yohsmapl}{\mathbf{H}}
\newcommand{\yohsmapr}{\mathbf{h}}

\newcommand{\yoidx}{\mathbf{Ind}}
\newcommand{\yoprmet}[1]{\mathrm{PrMet}(#1)}

\newcommand{\yomext}[1]{\widehat{#1}}

\newcommand{\yocompsp}[2]{\widetilde{#1}^{#2}}
\newcommand{\yocompdis}[1]{{#1}^{\diamondsuit}}

\newcommand{\yodbddsym}{\mathbf{Me}_{+}\!\mathbf{C}}
\newcommand{\yodbddsp}[2]{\yodbddsym\!\left(#1, #2\right)}

\newcommand{\yowapp}[2]{W_{#1, #2}}
\newcommand{\yomapp}[2]{M_{#1, #2}}

\newcommand{\yold}[1]{\mathbf{CCD}(#1)}

\newcommand{\yomsp}{Q}
\newcommand{\yomea}{\lambda}

\newcommand{\yokarimap}{\psi}
\newcommand{\yokarimapq}{\Psi}

\begin{document}

\title[An isometric  extensor of    metrics]{
An isometric  extensor of    metrics}

\author{Yoshito Ishiki}

\address[Yoshito Ishiki]{
Department of Mathematical Sciences
\endgraf
Tokyo Metropolitan University
\endgraf
Minami-osawa Hachioji Tokyo 192-0397
\endgraf
Japan}

\email{ishiki-yoshito@tmu.ac.jp}



\date{\today}

\begin{abstract}
In this paper, for a metrizable space $Z$, we consider the space of metrics that generate the same topology of $Z$, and that space of metrics is equipped with the supremum metrics. For a metrizable space $X$ and a closed subset $A$ of it, we construct a map $E$ from the space of metrics on $A$ into the space of metrics on $X$ such that  $E$ is an extensor  of metrics and preserves the supremum metrics between metrics. 
\end{abstract}

\maketitle

\tableofcontents

\section{Intoroduction}\label{sec:intro}
\subsection{Backgrounds}
For a metrizable space 
$Z$, 
we denote by 
$\met{Z}$
the 
set of all 
metrics on 
$Z$ 
that generate the same topology of 
$Z$. 
We write
$\oumetbd{Z}$
as the set of all bounded metrics in 
$\met{Z}$. 
We define the 
supremum 
metric
$\metdis_{Z}$
 on 
$\met{Z}$ by 
\[
\metdis_{Z}(d, e)=\sup_{x, y\in Z}|d(x, y)-e(x, y)|. 
\]
Take  a closed subset
$A$ 
of
 $Z$. 
It was Felix  Hausdorff 
\cite{Ha1930}
who first 
approached 
the extension problem of 
metrics on 
$A$ 
to 
$Z$, 
and proved that 
for each 
$d\in \met{Z}$, 
there exists
$D\in \met{Z}$
such that 
$D|_{A^{2}}=d$. 
Independently, 
in later years, 
Bing 
\cite{MR0024609} 
also proved the
same theorem
(see also 
\cite{MR0230285}, 
\cite{MR0049543}, 
and 
\cite{MR321026}). 
Based on  Hausdorff's result, 
several mathematicians 
explored 
developments of it. 
For variants of Hausdorff's extension theorem, 
see, for example,  
\cite{Ishiki2020int}, 
\cite{MR3135687}, 
\cite{MR2854677}. 
For extensions of ultrametrics
(non-Archimedean metrics), 
see 
\cite{tymchatyn2005note}, 
\cite{stasyuk2009continuous}, 
\cite{Ishiki2021ultra}, 
\cite{MR4527953}, 
\cite{Ishiki2022highpower}, 
and 
\cite{Ishiki2022factor}. 

Hausdorff's extension theorem 
only states that 
we can find  a extended metric
$D\in \met{X}$ 
of each  
$d\in \met{A}$. 
Thus, 
the next point of 
interest is 
whether these extensions of metrics 
can be taken simultaneously. 
In other words, 
can we construct 
an  extensor 
 $\met{A}\to \met{X}$
 that is continuous 
 with respect to 
  some topologies  on 
 spaces of metrics?
As long as the author knows, 
in 1981, 
Nguyen Van Khue and Nguyen To Nhu
\cite{NN1981}
 first
constructed 
simultaneous 
 extensions
$\Phi_{1}, \Phi_{2}\colon \met{A}\to \met{X}$, 
which satisfy that 
$\Phi_{1}$ 
is
$20$-Lipschitz 
with respect to 
the supremum 
metrics 
on spaces of metrics, 
and 
$\Phi_{2}$ 
is continuous with
 respect to the topologies of 
point-wise convergence
and preserving orders.

An extensor 
$\Phi\colon \oumetbd{A}\to \oumetbd{Z}$ 
on spaces of bounded metics 
is 
\textit{isometric}  
(with respect to the supremum 
metrics) 
 if 
we have 
$\metdis_{A}(d, e)=\metdis_{Z}(\Phi(d), \Phi(e))$
for all 
$d, e\in \met{A}$. 
Such  extensors on bounded metrics 
 have been obtained  
by 
Bessaga
 \cite{MR1211761}, 
 Banakh
 \cite{MR1811849}, 
 Pikhurko
 \cite{MR1692801}, 
and 
 Zarichnyi
 \cite{zarichnyi1996regular}.

As a non-Archimedean 
analogue of spaces of metrics, 
we can define 
the space $\ult{X}{R}$ of 
$R$-valued ultrametrics 
on 
$X$
and 
the non-Archimedean 
supremum metric 
$\umetdis_{X}^{R}$, 
where 
$R$ is a 
subset of $[0, \infty)$
with 
$0\in R$. 
In this paper, 
we omit the details of 
these concepts
(see 
\cite{Ishiki2021ultra} 
and 
\cite{Ishiki2022factor}). 
 For an ultrametrizable space 
 $X$, and a closed subset 
 $A$ 
 of
  $X$, 
 the author 
constructed 
an isometric 
extensor from 
$(\ult{A}{R}, \umetdis_{X}^{R})$
into 
$(\ult{X}{R}, \umetdis_{X}^{R})$
(see \cite[Theorem 4.7]{Ishiki2022factor}).

 In this paper, 
we shall 
construct 
an isometric extensor 
$\yomainmap\colon\met{A}\to \met{X}$
of metrics. 
Remark that, in contrast to the previous results, 
our extensor 
$\yomainmap$
can be applied  to
not only bounded metric but also  
 unbounded metrics, 
 and 
our main result  is 
an Archimedean
 analogue of 
\cite[Theorem 4.7]{Ishiki2022factor}.

\subsection{Main result}
For a metrizable space 
$X$, 
we denote by 
$\compm{X}$
the set of all complete metrics
$d\in \met{X}$. 
Of course, 
we have 
$\compm{X}\neq \emptyset$
if and only if 
$X$ is completely metrizable. 
We also 
denote by 
$\yocpm{X}$
the set of all 
pseudometrics 
$d\colon X\times X\to [0, \infty)$
that are continuous as maps between 
$X\times X$
 and 
 $[0, \infty)$. 
The set 
 $\yocpm{X}$
 is equipped with 
 the supremum metric 
 $\metdis_{X}$, 
 where 
 we use the same symbol 
 as the metric
 $\metdis_{X}$ on $\met{X}$. 
 By the definitions, 
we have 
$\met{X}\yosub\yocpm{X}$. 
The author showed that 
the set 
$\met{X}$
 is comeager in 
$\yocpm{X}$ 
(see \cite[Theorem 1.3]{Ishiki2024smbaire}). 

We prove the following 
generalization of 
 Hausdorff's metric extension 
theorem. 
\begin{thm}\label{thm:main1}
Let 
$X$ 
be 
a metrizable space, 
and 
$A$
 be a closed subset of 
 $X$. 
 Then there exists a map 
 $\yomainmap\colon \met{A}\to \met{X}$
 such that 
 \begin{enumerate}[label=\textup{(\arabic*)}]
 \item\label{item:m1ext}
for every 
$d\in \met{A}$
 we have 
$\yomainmap(d)|_{A^{2}}=d$; 
\item\label{item:m1isom}
the map 
$\yomainmap$
  is 
an isometric embedding, 
i.e., 
for every pair  
$d, e\in \met{A}$, 
we have 
\[
\metdis_{A}(d, e)=\metdis_{X}(\yomainmap(d),
\yomainmap(e));
\]
\item\label{item:m1bounded}
we have 
$\yomainmap(\oumetbd{A})\yosub 
\oumetbd{X}$;

\item\label{item:m1comp}
we have 
$\yomainmap(\compm{A})\yosub \compm{X}$. 
 \end{enumerate}
Furthermore,
 we also obtain an isometric extensor 
 $\yomainmapf\colon \yocpm{A}\to \yocpm{X}$
 of pseudometrics
such that 
$\yomainmapf|_{\met{A}}=\yomainmap$. 
In this setting, we have 
$\yomainmapf(\met{A})=\yomainmapf(\yocpm{A})\cap \met{X}$. 
Thus, the image 
$\yomainmap(\met{A})$ is 
closed in 
$\met{X}$. 
\end{thm}

Our proof of
the main result is 
based on  
the idea of extending spaces and
extending homeomorphisms, 
which was used in 
\cite{zarichnyi1996regular},  
and can go back to 
\cite{Ha1930}, 
\cite{Hausdorff1938}, 
and 
\cite{Ku1938}. 
In other words, 
for a metric space 
$Z$
and for a closed subset 
$A$ 
of 
$Z$, 
we will find an extension  metrizable space 
$L$ 
of 
$A$
with
a
 topological embedding  
$I\colon A\to L$
such that every metric 
$d\in \met{A}$
can be naturally 
 extended to a metric on 
 $L$
  through 
 $I$. 
In this situation, we
will construct 
an topological embedding 
$J\colon Z\to L$ 
such that 
$J|_{A}=I$ and 
define 
an extended metric 
$\yomainmap(d)$
by  a pullback metric
$E(d)=J^{*}d$
 induced by 
$J$. 

The organization of this paper is as follows: 
In 
Section
 \ref{sec:pre}, 
we review several basic concepts on metrics spaces. 
We also explain 
the constructions of 
the
 $\ell^{1}$-products, 
the 
$1$-Wasserstein spaces, 
and 
spaces of measurable maps. 
For example, 
we give a 
characterization 
of the topology of 
$1$-Wasserstein space in 
a similar manner of the Portmanteau theorem even 
if an underlying space is not assumed to be complete. 
In the end of 
Section 
\ref{sec:pre}, 
we summarize these constructions 
as an 
\textit{osmotic construction},  
which 
 is a method to obtain 
extension spaces of a metrizable space
where metrics on given spaces are naturally extended. 
The author hopes that this notion would 
 be helpful to 
 improve  
 our main result in the future. 
In 
Section
 \ref{sec:wd}, 
we review the classical 
discussion called the 
Whitney--Dugundji
decomposition, 
which plays a key 
role of 
the proof of 
the main result. 
The whole of 
Section 
\ref{sec:proof}
is devoted to the proof of 
Theorem 
\ref{thm:main1}. 
Section 
\ref{sec:ques}
exhibits 
several questions on 
extensors  of metrics. 
\begin{ac}
This work was supported by JSPS 
KAKENHI Grant Number 
JP24KJ0182. 

The author would like to 
thank Tatsuya Goto for the advices on
properties of  Borel sets. 
\end{ac}


\section{Preliminaries}\label{sec:pre}
\subsection{Basic notations}
First, we review the 
basic notions and notations on metric spaces. 
For a metric space 
$(Z, w)$, 
we denote by 
$U(x, r; w)$
the open ball centered at 
$x\in Z$
with radius 
$r\in (0, \infty)$. 
For a subset 
$A\yosub Z$, 
and for
 $x\in Z$
we define 
$w(x, A)=\inf_{a\in A}w(x, a)$. 
Note that the function 
$x\mapsto w(x, A)$ 
is
$1$-Lipschitz. 
For a subset 
$S$
of 
$Z$, 
we represent 
$\yodiam_{w}(S)$
the diameter of 
$S$ 
with respect to 
$w$. 

In this 
paper, 
we often denote by 
$\yoomega$
the set of all non-negative integers
when we regard  the set of non-negative integers as 
a (discrete) space. 
Of course, 
we have 
$\yoomega=\zz_{\ge 0}$
 as a set. 
When emphasizing 
the set of 
 integers as an index set, 
we rather use 
$\zz_{\ge 0}$
than 
$\yoomega$. 

\subsection{Components of spaces of metrics}\label{subsec:compo}
For a metrizable space 
$Z$, 
and for 
$d, e\in \met{Z}$, 
we write 
$d\sim e$
 if 
$\metdis_{Z}(d, e)<\infty$. 
Then 
``$\sim$''
becomes an 
equivalence relationship 
on 
$\met{Z}$. 
We represent 
$\yoslamet{Z}=\met{Z}/\!\!\sim$.
For a member 
$d\in\met{Z}$, 
we write 
$\yocclq{d}$ 
as
the equivalence  class of 
$d$. 
Remark that 
each 
$\yoccl \in \yoslamet{Z}$
is 
a
(path-)connected component of 
$\met{Z}$, and 
it 
is also a clopen subset of 
$\met{Z}$. 

In the proof of main theorem, 
first, 
we fix 
a metric $m\in \met{A}$ and
consider 
the 
equivalence class 
$\yoccl=\yocclq{m}\in \yoslamet{A}$. 
Second,   we 
construct
an isometric map 
$\yomainmap\colon \yoccl\to \met{X}$, 
and 
gluing then together,   
we obtain 
$\yomainmap\colon \met{A}\to \met{X}$.


\subsection{Spaces of maps with finite supports}\label{subsec:finsp}
For
a set 
$T$, 
we denote by 
$\yofmosp{T}$
the set of 
all maps 
$f$ 
from 
$T$ 
into 
$\rr$
such that 
$\{\, s\in T\mid f(s)\neq 0\, \}$
is finite. 
Let 
$\yofmodisn$ 
stand for  the supremum metric on 
$\yofmosp{T}$. 
Namely, 
we have 
$\yofmodisn(f, g)=\max\{\, |f(s)-g(s)|\mid s\in T\, \}$.

\subsection{Constructions of metric spaces}\label{subsec:constmet}
In this section, 
we shall
 introduce 
 three constructions of metric spaces, 
the 
$\ell^{1}$-product 
with a fixed metric space, 
the 
$1$-Wasserstein space, 
and 
the space of measurable functions.

\subsubsection{$\ell^{1}$-product of  spaces}\label{subsec:ell1}
Fix a metric space 
$(S, u)$. 
For a metrizable space 
$Z$, 
we define
$\yoosmisasp{Z, S}
=Z\times S$. 
When  
$S$ 
is clear by the context, 
we simply write 
$\yoosmisasp{Z}=\yoosmisasp{Z, S}$. 
For a metric 
$d\in \met{Z}$, 
we define 
$\yoosmisadis{d}=d\times_{\ell^{1}} u$, 
i.e., 
$\yoosmisadis{d}((x, s), (y, t))=
d(x, y)+u(s, t)$. 
Fix a point 
$o\in S$, 
we define 
$\yoosmisaemb{Z}\colon Z\to \yoosmisasp{Z}$
by 
$\yoosmisaemb{Z}(x)=(x, o)$. 

\begin{prop}\label{prop:ellisom}
Let  
$(S, u)$ 
be a fixed metric space, 
 $o\in S$ be a fixed point, 
and 
$Z$ 
be  
a metrizable  space. 
We also fix 
$\yoccl\in \yoslamet{Z}$. 
Then 
the following statements 
are true:
\begin{enumerate}[label=\textup{(\arabic*)}]
\item\label{item:osaext}
For every 
$d\in \yoccl$, 
and 
for every pair 
$x, y\in Z$, 
we have 
\[
\yoosmisadis{d}\left(\yoosmisaemb{Z}(x), \yoosmisaemb{Z}(y)\right)=
d(x, y). 
\]
\item\label{item:osaisom}
The map 
$\yoosmisa\colon \met{Z}\to \met{\yoosmisasp{Z}}$
by 
$d\mapsto \yoosmisadis{d}$
is an isometric embedding, 
i.e., 
for every pair 
$d, e\in \met{Z}$, 
 we have 
\[
\metdis_{Z}(d, e)
=\metdis_{\yoosmisasp{Z}}(\yoosmisadis{d}, \yoosmisadis{e}). 
\]

\item\label{item:osmisatop}
For every pair 
$d, e\in \yoccl$, 
the metrics
$\yoosmisadis{d}$ 
and 
$\yoosmisadis{e}$
generate the 
same topology of 
$\yoosmisasp{Z}$.

\end{enumerate}
\end{prop}
%
\begin{proof}
Statement 
\ref{item:osaext}
 follows from 
 the definitions of 
 $\yoosmisadis{d}$
 and 
 $\yoosmisaemb{Z}(x)$. 
 Now we show 
 \ref{item:osaisom}. 
For every pair 
$(x, s), (y, t)\in \yoosmisa(Z)$, 
we have 
\begin{align*}
&|\yoosmisadis{d}((x, s), (y, t))
-
\yoosmisadis{e}((x, s), (y, t))|
\\
&=
|d(x, y)+u(s, t)
-e(x, y)-u(s, t)|
=|d(x, y)-e(x, y)|. 
\end{align*}
Thus we have 
$\metdis_{Z}(d, e)
=\metdis_{\yoosmisasp{Z}}(\yoosmisadis{d}, \yoosmisadis{e})$. 

Statement \ref{item:osmisatop}
is trivial. 
Moreover, 
for every metric $d\in \met{Z}$, 
the metric 
$\yoosmisadis{d}$
generates the product topology of 
$\yoosmisasp{Z}=Z\times S$. 
This statement is 
just a preparation for introducing 
osmotic constructions. 
This completes the proof. 
\end{proof}

\subsubsection{$1$-Wasserstein spaces}
Since we 
will consider the 
Wasserstein space
on a metrizable space that is not assumed to be 
separable, 
we follow the 
construction 
of Wasserstein spaces 
discussed in 
\cite{MR2861765}
and 
\cite{MR3858004}
using Radon measures. 

Let
 $Z$
  be 
 a metrizable space. 
 A Borel measure
 $\mu$
  on
  $Z$
 is said to be 
 \yoemph{Radon}
  if 
 for every Borel subset 
 $A$ 
 of 
 $Z$, 
 and for every 
 $\epsilon \in (0, \infty)$
there exists a compact subset 
$K$ 
of 
$Z$ such that 
$K\yosub A$
and 
$\mu(A\setminus K)<\epsilon$. 
 We denote by 
 $\yowasspw{Z}$
 the set of all 
 Radon 
 probability 
 measures on 
 $Z$. 
 We define 
 $\yowasemb{Z}\colon Z\to \yowasspw{Z}$
 by 
 $\yowasemb{Z}(x)=\yodmas{x}$, 
 where 
 $\yodmas{x}$
 is the Dirac measure on 
 $x$. 
 
 For 
 $\alpha, \beta\in \yowasspw{Z}$, 
 we 
 denote by 
 $\Pi(\alpha, \beta)$
 the set of 
 all 
 $\pi\in \yowasspw{Z\times Z}$
 such that 
 $\pi(A\times Z)=\alpha(A)$
 and $\pi(Z\times A)=\beta(A)$ for 
 all Borel subsets $A$
 of 
 $Z$. 
Now  we define 
the $1$-Wasserstein distance 
$\yowasdis{d}$
with respect to 
$d$
by 
\[
\yowasdis{d}(\alpha, \beta)
=\inf_{\pi\in \Pi(\alpha, \beta)}\int_{Z\times Z}d(s, t)\ d\pi(s, t), 
\]
where 
$\alpha, \beta\in \yowasspw{Z}$.
It should be noted that, 
in general, 
$\yowasdis{d}$ 
can take 
the value 
$\infty$. 
Thus, 
we focus on the subset of 
$\yowasspw{Z}$ 
where 
$\yowasdis{d}$ 
takes
only finite values. 
Fix an equivalence class 
 $\yoccl \in \yoslamet{Z}$. 
We denote by 
$\yowasspc{Z}{\yoccl}$
the set of 
$\alpha\in \yowasspw{Z}$
such that 
$\yowasdis{d}(\alpha, \yodmas{p})<\infty$
for some 
$p\in Z$. 
This condition is also
equivalent to the inequality 
$\yowasdis{d}(\alpha, \yodmas{q})<\infty$
for every 
$q\in Z$. 
By the definition, 
the set 
$\yowasspc{Z}{\yoccl}$
does not depend on 
the choice of a representative of 
the class 
$\yoccl$; 
namely, 
for every pair 
$d, e\in \yoccl$, 
and for every 
$\alpha \in \yowasspw{Z}$, 
we have 
$\yowasdis{d}(\alpha, \yodmas{p})<\infty$
if and only if 
$\yowasdis{e}(\alpha, \yodmas{p})<\infty$. 
When 
we  fix a class 
$\yoccl$, 
we simply write
$\yowassp{Z}=\yowasspc{Z}{\yoccl}$. 
We call 
$(\yowassp{Z}, \yowasdis{d})$
the 
\yoemph{$1$-Wasserstein space on
$(Z, d)$}. 

The 
$1$-Wasserstein 
space is  sometimes 
called 
the 
\yoemph{Lipschitz-free space} with 
emphasizing the aspect of the dual space of 
the space of Lipschitz functions through 
the 
 Kantorovich--Rubinstein duality. 
This space is also sometimes referred 
to 
as
the 
\yoemph{Arens--Eells space}, 
named  after
the 
Arens--Eells
embedding theorem  
\cite{MR0081458}.

From now on, 
we review properties of 
the 
$1$-Wasserstein spaces which 
we will use in the present paper.

For a metric sapce 
$(X, d)$, 
we denote by 
$\yolipsp{X, d}$
the set of 
all 
real-valued 
$1$-Lipschitz functions
on $(X, d)$. 
To compute 
 values of Wasserstein distance 
between  specific measures, 
we shall make use of 
the Kantorovich--Rubinstein duality. 
\begin{thm}[The Kantorovich--Rubinstein duality]\label{thm:krdual}
Let 
$(Z, d)$ 
be a metric space. 
Then for every pair 
$\alpha, \beta\in \yowassp{Z}$, 
we have 
\[
\yowasdis{d}(\alpha, \beta)
=
\sup\left\{\, \int_{Z}f(x)\ d(\alpha-\beta)(x) \  \middle|\  f\in \yolipsp{X, d}\, \right\}. 
\]
\end{thm}
\begin{proof}
Remark that we does not assume that 
$X$ 
is Polish in the statement
of Theorem \ref{thm:krdual}. 
For the proof, 
we refer the readers to 
\cite[Theorem 4.1]{MR2861765}
and 
\cite[Theorem 1]{MR0867434}, 
or 
we can prove this generalized duality using 
the duality for Polish spaces 
together with 
\cite[Lemma 2.5]{MR3858004}, 
the McShane--Whitney extension of Lipschitz functions, 
and the definition of Radon 
probability measures. 
\end{proof}

To show the main result, 
we need the following 
formulae of 
Wasserstein distances.

\begin{lem}\label{lem:wasfinfin}
Let 
$Z$ 
be a metrizable space, 
and 
fix 
$\yoccl\in \yoslamet{Z}$
and 
$n\in \zz_{\ge 0}$. 
Take 
 two 
sequences 
$s_{1}, \dots ,s_{n}$
and 
$t_{1}, \dots, t_{n}$
of non-negative  reals 
 such that 
$\sum_{i=1}^{n}s_{i}
=\sum_{i=1}^{n}t_{i}=1$, 
and take 
a finite sequence
 $x_{1}, \dots, x_{n}$ 
 in 
$X$. 
Then  
for every  point 
$b\in Z$, 
we have 
\[
\yowasdis{d}\left(
\sum_{i=1}^{n}s_{i}\yodmas{x_{i}}, 
\sum_{i=1}^{n}t_{i}\yodmas{x_{i}}, 
\right)
\le  \sum_{i=1}^{n}|s_{i}-t_{i}|d(x_{i}, b). 
\]
\end{lem}
%
\begin{proof}
Put 
$\alpha=\sum_{i=1}^{n}s_{i}\yodmas{x_{i}}$,
and  
$\beta=\sum_{i=1}^{n}t_{i}\yodmas{x_{i}}$. 
Then 
for every 
$1$-Lipschitz map  
$f\colon (Z, d)\to (\rr, |*|)$, 
we have 
\begin{align*}
&\int_{Z}f(x)\ d(\alpha-\beta)
=
\sum_{i=1}^{n}s_{i}f(x_{i})-
\sum_{i=1}^{n}t_{i}f(x_{i})\\
&=
\sum_{i=1}^{n}s_{i}f(x_{i})-
f(b)+f(b)-
\sum_{i=1}^{n}
t_{i}f(x_{i})\\
&=
\sum_{i=1}^{n}s_{i}f(x_{i})
-
\sum_{i=1}^{n}s_{i}f(b)
+
\sum_{i=1}^{n}t_{i}f(b)-
\sum_{i=1}^{n}t_{i}f(x_{i})
\\
&=
\sum_{i=1}^{n}s_{i}(f(x_{i})-f(b))
+
\sum_{i=1}^{n}t_{i}(f(b)-f(x_{i}))
\\
&=
\sum_{i=1}^{n}(s_{i}-t_{i})(f(x_{i})-f(b)).
\end{align*}
Here, 
for each 
$i\in \zz_{\ge 0}$,
we take 
 a 
number 
$\mathbf{sgn}_{i}\in \{1, -1\}$
such that 
$\mathbf{sgn}_{i}\cdot |s_{i}-t_{i}|=(s_{i}-t_{i})$. 
Then we can continue to compute. 
\begin{align*}
&\sum_{i=1}^{n}(s_{i}-t_{i})(f(x_{i})-f(b))
=
\sum_{i=1}^{n}
|s_{i}-t_{i}|\cdot \mathbf{sgn}_{i}\cdot (f(x_{i})-f(b))
\\
&\le 
\sum_{i=1}^{n}
|s_{i}-t_{i}|\cdot |f(x_{i})-f(b)|
\le \sum_{i=1}^{n}
|s_{i}-t_{i}|\cdot d(x_{i}, b).  
\end{align*}
Thus, 
for every $f\in \yolipsp{Z, d}$, we have 
\[
\int_{Z}f(x)\ d(\alpha-\beta)
\le
\sum_{i=1}^{n} |s_{i}-t_{i}|\cdot d(x_{i}, b). 
\]
Therefore, 
the Kantorovich--Rubinstein 
duality
(Theorem \ref{thm:krdual})
implies the lemma. 
\end{proof}

\begin{lem}\label{lem:wasfin1}
Let 
$Z$ 
be a
metrizable space, 
and fix 
$\yoccl\in \yoslamet{Z}$. 
Take a finite 
sequence  
$c_{1}, \dots , c_{n}$
of 
non-negative  reals 
 such that 
$\sum_{i=1}^{n}c_{i}=1$, 
and take 
a
 finite sequence 
 $x_{1}, \dots, x_{n}$ 
 in 
$X$
Then 
for every 
point 
$p\in X$, 
we have 
\[
\yowasdis{d}\left(\sum_{i=1}^{n}c_{i}\yodmas{x_{i}}, \yodmas{p}\right)
= \sum_{i=1}^{n}c_{i}d(x_{i}, p). 
\]
\end{lem}
%
\begin{proof}
Put 
$\alpha=\sum_{i=1}^{n}c_{i}\yodmas{x_{i}}$,
and
$\beta=\yodmas{p}$. 
Define
 a probability measure 
$\mu\in \Pi(\alpha, \beta)$
on $Z\times Z$ by 
$\mu=\sum_{i=1}^{n}c_{i}\yodmas{(x_{i}, p)}$. 
Under this situation, 
we have 
\[
\yowasdis{d}(\alpha, \beta)\le 
\int_{Z\times Z}d(x, y)\ d\mu(x, y)
=\sum_{i=1}^{n}c_{i}d(x_{i}, p). 
\]
To obtain the opposite inequality, 
define 
$f\in  \yolipsp{X, d}$
by 
$f(x)=d(x, p)$. 
Then 
we have 
\begin{align*}
&\int_{Z}f\ d(\alpha-\beta)
=
\left(\sum_{i=1}^{n}c_{i}f(x_{i})\right)
-f(p)
=
\left(\sum_{i=1}^{n}c_{i}d(x_{i}, p)\right)
-d(p, p)
\\
&=
\sum_{i=1}^{n}c_{i}d(x_{i}, p). 
\end{align*}
Thus 
the Kantorovich--Rubinstein 
duality
(Theorem \ref{thm:krdual})
 implies  that 
$\sum_{i=1}^{n}c_{i}d(x_{i}, p)\le \yowasdis{d}(\alpha, \beta)$. 
This finishes the proof. 
\end{proof}

\begin{cor}\label{cor:oswasext}
Let 
$Z$ 
be a metrizable space, 
and fix a class 
$\yoccl\in \yoslamet{Z}$.
Then for every 
$d\in \yoccl$, 
and for every pair 
$x, y\in Z$, 
we have 
\[
\yowasdis{d}\left(\yowasemb{Z}(x), \yowasemb{Z}(y)\right)
=d(x, y). 
\]
\end{cor}
\begin{proof}
This lemma is a
special case of 
Lemma \ref{lem:wasfin1}. 
Recall that 
$\yowasemb{Z}(x)=\yodmas{x}$. 
\end{proof}

From now on, 
we consider the topology of 
$1$-Wasserstein
 spaces of Radon measures 
on metric spaces that are not necessarily 
complete.

We first introduce the 
known 
description of 
$1$-Wasserstein spaces on 
complete spaces. 
For a metric space 
$(Z, d)$, 
we denote by 
$\yodbddsp{Z}{d}$
the set of all 
$f\colon Z\to \rr$ 
such that 
there exist 
$A, B\in (0, \infty)$ 
and 
$p\in Z$
for which 
$|f(x)|\le A\cdot d(x, p)+B$
 for all 
 $x\in Z$. 
 We do not impose the continuity of 
 each 
 $f\in \yodbddsp{Z}{d}$. 
 The symbol  
 ``$\yodbddsym$''
 means 
 ``Metric  + Constant''. 
 Note that 
 $f\in \yodbddsp{Z}{d}$
  if and only if 
 for every 
 $q\in Z$,  
 there exist
 $\widetilde{A}, \widetilde{B}\in (0, \infty)$ 
 such that 
 $|f(x)|\le \widetilde{A}\cdot d(x, q)+\widetilde{B}$
 for all 
 $x\in Z$.

The following theorem describes the 
topology of 
$1$-Wasserstein space on 
a complete metric space. 
As Theorem 
\ref{thm:wastop20240918}, 
we will 
remove the assumption that an underlying space is complete from 
Theorem \ref{thm:wastop}. 
\begin{thm}\label{thm:wastop}
Let 
$(R, w)$ 
be a complete metric space, 
and let 
$\{\mu_{i}\}_{i\in \zz_{\ge 0}}$
be a sequence 
in 
$\yowassp{R}$
and 
take 
$\mu\in \yowassp{R}$. 
Then the following statements are equivalent to 
each other. 
\begin{enumerate}[label=\textup{(\arabic*)}]
\item\label{item:wconv}
The sequence 
$\{\mu_{i}\}_{i\in \zz_{\ge 0}}$
converges to
 $\mu$
 in 
 $(\yowassp{R}, \yowasdis{w})$.


\item\label{item:bddconv}
If a continuous function 
$f\colon R\to \rr$
belongs to 
$\yodbddsp{R}{w}$, 
then we have 
\[
\int_{R}f(x)\ d\mu_{i}
\to \int_{R}f(x)\ d\mu
\]
 as 
$i\to \infty$. 
\end{enumerate}
\end{thm}
%
\begin{proof}
See 
\cite[Theorem 2.11 and Theorem 2.12]{MR3858004}. 
\end{proof}

In the next proposition, 
it is shown that 
every 
Radon probability measure on 
$Z$
has a $\sigma$-compact support, 
and hence is 
can be extended to a 
measure on an extension space 
of $Z$. 
This observation will  build a bridge between 
the topologies of 
$1$-Wasserstein spaces  on 
complete and incomplete spaces. 
\begin{prop}\label{prop:measureksigma}
Let 
$(Z, d)$
be a metric space, and 
$\mu\in \yowasspw{Z}$. 
Then there exists 
a subset 
$S$ 
of
$Z$ 
such that 
\begin{enumerate}[label=\textup{(m\arabic*)}]
\item 
the set 
$S$ 
is 
$\sigma$-compact. 
In particular, 
the set 
$S$ 
is 
absolutely 
 $F_{\sigma}$;
\item 
we have 
$\mu(S)=1$. 
\end{enumerate}
\end{prop}
%
\begin{proof}
For each 
$n\in \zz_{\ge 0}$, 
due to the definition of 
Radon
probability  measures, 
we obtain 
a compact subset 
$K_{n}$ of 
$Z$ such that 
$\mu(K_{n})>1-2^{-n}$. 
Put 
$S=\bigcup_{n\in \zz_{\ge 0}}K_{n}$. 
Then 
$S$ 
satisfies the 
conclusion of the proposition. 
\end{proof}

\begin{df}\label{df:extmeasure}
Take a subset 
$S$ 
mentioned in 
 Proposition 
\ref{prop:measureksigma}, 
and 
take  a metrizable space 
$\Omega$ 
with 
$Z\yosub \Omega$. 
Then 
$S$ 
is also 
Borel in 
$\Omega$
 since it is 
 $\sigma$-compact. 
Then,
 for every 
Borel set 
$A$ 
of 
$\Omega$, 
the intersection 
$A\cap S$ 
is Borel in 
$\Omega$. 
Thus it is also Borel in 
$S$. 
Since 
$S$ 
is Borel in 
$Z$, 
we can conclude that 
$A\cap S$
 is Borel in 
$Z$
(see also \cite[Proposition 3.1.9]{MR1619545}). 
Thus, 
 we can  define 
$\yomext{\mu}$ 
by 
$\yomext{\mu}(A)=\mu(A\cap S)$
for every Borel set 
$A$ 
of 
$\Omega$, 
and 
we see that
 $\yomext{\mu}$ 
also  becomes 
a
Radon Borel  measure on 
$\Omega$. 
Note that 
$\yomext{\mu}$
does not depend on 
the choice of 
$S$. 
Based on this phenomenon, 
whenever we are  given  a metrizable space 
$\Omega$
 such that 
 $Z\yosub \Omega$, 
we use the same symbol 
$\yomext{\mu}$
to denote 
  the extension 
 of 
 $\mu$ 
 constructed above. 
 \end{df}

\begin{prop}\label{prop:wawasamesame}
Let 
$(Z, d)$ 
be a metric space
and 
let 
$(\yocompsp{Z}{d}, \yocompdis{d})$
 denote  the completion of 
 $(Z, d)$. 
Then we have 
\[
\yowasdis{d}(\alpha, \beta)
=
\yowasdis{\yocompdis{d}}(\yomext{\alpha}, \yomext{\beta}). 
\]
\end{prop}
%
\begin{proof}
By the Kantorovich--Rubinstein duality
(Theorem \ref{thm:krdual}), 
for each 
$\epsilon \in (0, \infty)$, 
there exists 
$f\in \yolipsp{X, d}$
such that 
\[
\yowasdis{d}(\alpha, \beta)-\epsilon<\int_{Z}f\ d(\alpha-\beta).
\] 
Since 
$f$ 
is Lipschitz
and 
$(\yocompsp{Z}{d}, \yocompdis{d})$ 
is the completion of 
$(Z, d)$, 
using 
Cauchy sequences 
(see also 
\cite[Theorem 2]{MR0390999} and \cite{MR0969516}), 
we can obtain 
$F\in \yolipsp{\yocompsp{Z}{d}, \yocompdis{d}}$
such that 
$F|_{Z}=f$. 
Since 
 $\yomext{\alpha}(\yocompsp{Z}{d}\setminus Z)=0$
 and 
 $\yomext{\beta}(\yocompsp{Z}{d}\setminus Z)=0$
 (see Definition  \ref{df:extmeasure}), 
using 
the Kantorovich--Rubinstein duality
(Theorem \ref{thm:krdual}) again,   
we have 
 \[
 \int_{Z}f\ d(\alpha-\beta)= 
 \int_{\yocompsp{Z}{d}}F\ d(\yomext{\alpha}-\yomext{\beta})
 \le \yowasdis{\yocompdis{d}}(\yomext{\alpha}, \yomext{\beta}). 
 \]
 Thus 
 $\yowasdis{d}(\alpha, \beta)\le 
 \yowasdis{\yocompdis{d}}(\yomext{\alpha}, \yomext{\beta})$. 
 Next, 
 we show the opposite inequality. 
 For each 
 $\epsilon \in (0, \infty)$, 
 take 
 $\phi\in \yolipsp{\yocompsp{Z}{d}, \yocompdis{d}}$ 
 such that 
 \[
 \yowasdis{\yocompdis{d}}(\yomext{\alpha}, \yomext{\beta})-\epsilon 
 \le \int_{\yocompsp{Z}{d}}\phi\ d(\yomext{\alpha}-\yomext{\beta}).
 \]
 Of course, we obtain 
 $\phi|_{X}\in \yolipsp{Z, d}$, 
 and 
\[
\int_{\yocompsp{Z}{d}}\phi\ d(\yomext{\alpha}-\yomext{\beta})=
 \int_{Z}\phi|_{X}\ d(\alpha-\beta)\le 
 \yowasdis{d}(\alpha, \beta).
 \] 
 Hence 
  $\yowasdis{d}(\alpha, \beta)\le 
 \yowasdis{\yocompdis{d}}(\yomext{\alpha}, \yomext{\beta})$. 
 Therefore, 
  we conclude that 
 $\yowasdis{d}(\alpha, \beta)=
 \yowasdis{\yocompdis{d}}(\yomext{\alpha}, \yomext{\beta})$. 
\end{proof}

\begin{rmk}
Proposition 
\ref{prop:wawasamesame}
indicates that 
the extension 
$\yomext{\mu}$
of 
$\mu$
can be obtained as follows:
For a metrizable space
 $X$, 
let
 $\yold{X}$ 
 be the set of all 
finitely convex combination of 
the  Dirac measures on
 $X$. 
We regard 
 $\yold{X}$  as a subspace of 
$\yowassp{X}$. 
Note that, in general, 
the subset 
 $\yold{X}$ 
 is dense in 
$(\yowassp{X}, \yowasdis{d})$
(This
statement 
for a complete space
 $X$
  follows from 
\cite[Theorem 2.7]{MR3858004}.
Using Lemma 
\ref{lem:wasfinfin}, 
we can also show 
the case of an incomplete space 
$X$). 
Since 
$Z$
 is dense in 
$\yocompsp{Z}{d}$, 
we obtain an isometric  embedding 
$\yokarimap\colon \yold{Z}\to \yold{\yocompsp{Z}{d}}$
such that 
$\yokarimap(\yodmas{z})=\yodmas{z}$
for all
 $z\in Z$, 
where 
$\yodmas{z}$
in the right hand side 
is a measure on 
$\yocompsp{Z}{d}$. 
In this setting, 
the image set 
$\yokarimap(\yold{Z})$
 is dense in 
$\yold{\yocompsp{Z}{d}}$. 
Thus we obtain an isometric embedding 
$\yokarimapq\colon \yowassp{Z}\to \yowassp{\yocompsp{Z}{d}}$
such that 
$\yokarimapq|_{\yold{Z}}=\yokarimap$. 
Then we can observe 
that 
$\yokarimapq(\mu)=\yomext{\mu}$. 
\end{rmk}

\begin{prop}\label{item:pointapproxlip}
Let 
$(Z, d)$
be a metric space, 
fix
$\yoccl\in \yoslamet{Z}$, 
take  
$e\in \yoccl$, 
and let
$\mu$
be a 
Radon 
probability measure on 
$Z$
 such that 
$\int_{Z}d(x, p)\ d\mu<\infty$. 
Let 
$(\yocompsp{Z}{d}, \yocompdis{d})$
denote the 
completion of 
$(Z, d)$. 
If a continuous function 
$f\colon Z\to \rr$
belongs to 
$\yodbddsp{Z}{d}$, 
then, 
for every 
$\epsilon\in (0, \infty)$, 
there exist Lipschitz functions 
$M, W\colon \yocompsp{Z}{d}\to \rr$ 
on 
$(\yocompsp{Z}{d}, \yocompdis{d})$
such that 
\begin{enumerate}[label=\textup{(\arabic*)}]
\item 
the functions 
$M$ 
and 
$W$
belong to 
$\yodbddsp{\yocompsp{Z}{d}}{\yocompdis{d}}$; 
\item
for every 
$x\in Z$, 
we have 
\[
M(x)\le f(x)\le W(x);
\]
\item 
we have 
\[
\int_{Z}W(x)-f(x)\ d\mu(x) <\epsilon, 
\]
and 
\[
\int_{Z}f(x)-M(x)\ d\mu(x) <\epsilon. 
\]
\end{enumerate}
\end{prop}
%
\begin{proof}
With respect to 
this proof, 
the author was inspired by  the proof of 
\cite[Subsection 8.5]{lebanon2012probability}. 
We make use of the 
McShane--Whitney extension 
(see 
\cite{MR1562984} 
and 
\cite{MR1501735}, 
see also 
 \cite[p.162]{MR0352214}). 
That extension is  sometimes effective 
for non-Lipschitz functions in some sense
(see, for example, \cite[Lemma 5.4]{MR3445278}).

For each 
$r\in (0, \infty)$, 
for each
 $x\in \yocompsp{Z}{d}$, 
we define 
\[
\yowapp{f}{r}(x)=\sup
\left\{\, f(q)-r\cdot \yocompdis{d}(x, q)\mid q\in Z\, \right\}
\]
and 
\[
\yomapp{f}{r}(x)=\inf 
\left\{\, f(q)+r\cdot \yocompdis{d}(x, q)\mid q\in Z\, \right\}. 
\]
Note that 
$\yomapp{f}{r}=-\yowapp{-f}{r}$. 
First, 
we
shall 
 show that
for a sufficiently 
 large 
$r\in (0, \infty)$, 
we have 
$\yowapp{f}{r}(x)<\infty$ 
and 
$-\infty<\yomapp{f}{r}(x)$. 
Since 
$f\in \yodbddsp{Z}{e}$, 
there exist 
$A, \tilde{B}\in (0, \infty)$
and 
$p\in Z$
satisfying  that 
$|f(x)|\le A\cdot e(x, p)+\widetilde{B}$ 
for all 
$x\in Z$. 
On account of 
$\metdis_{Z}(d, e)<\infty$, 
 we also  have 
 the inequality 
$|f(x)|\le A\cdot d(x, p)+B$ 
for all
 $x\in Z$, 
 where 
$B=\widetilde{B}+\metdis_{Z}(d, e)$. 
Take 
$r\in (0, \infty)$ 
so that 
$A\le r$. 
Under this stuiation, for every 
$q\in Z$, we obtain 
\begin{align*}
&f(q)-r\cdot \yocompdis{d}(x, q)
\le A\cdot d(q, p)+B-A\cdot \yocompdis{d}(x, q)\\
&
=B+A(\yocompdis{d}(q, p)-\yocompdis{d}(q, x))\le B+A\cdot \yocompdis{d}(x, p). 
\end{align*}
Thus, 
we have 
\[
\yowapp{f}{r}(x)\le A\cdot \yocompdis{d}(x, p)+B<\infty.
\] 
Since 
$\yomapp{f}{r}=-\yowapp{-f}{r}$,  and 
since 
$|-f(x)|=|f(x)|\le A\cdot d(x, p)+B$
for all 
$x\in Z$, 
we also have 
\[
\yomapp{f}{r}(x)\ge -A\cdot  \yocompdis{d}(x,  p)-B>-\infty.
\] 
As a result, 
we obtain two functions 
$\yomapp{f}{r}, \yowapp{f}{r}\colon \yocompsp{Z}{d}\to \rr$
such that 
$\yomapp{f}{r}(x)\le f(x)\le \yowapp{f}{r}(x)$
for all 
$x\in Z$.

In the same  way as  
 \cite[Theorem 2.1]{MR4103879}, 
 we see that 
 $\yomapp{f}{r}$ 
 and 
 $\yowapp{f}{r}$ 
 are
$r$-Lipschitz 
on
 $(\yocompsp{Z}{d}, \yocompdis{d})$, 
 and hence 
$\yomapp{f}{r}, \yowapp{f}{r}\in \yodbddsp{\yocompsp{Z}{d}}{\yocompdis{d}}$. 

Using 
a similar  method to
 \cite{MR1544388}, 
 and 
 \cite[Theorem 2.1 and Proposition 2.2]{MR4103879}, 
 we will show that, 
for each 
$x\in Z$, 
the values 
 $\yomapp{f}{r}(x)$ 
 and 
 $\yowapp{f}{r}(x)$ 
 convergence 
 to 
 $f(x)$  
 as 
 $r\to \infty$. 
 First we 
 deal with 
 the function 
 $\yowapp{f}{r}$.
  Take 
  an arbitrary 
  number 
  $\epsilon \in (0, \infty)$. 
  Since $f$ is continuous at 
  $x$, we can 
  find 
  $\delta\in (0, \infty)$ such that 
  if $d(x, q)<\delta$, 
  then 
$|f(x)-f(q)|\le \epsilon$. 
Take $k\in (0, \infty)$ so that 
\[
\delta\cdot k>\sup_{a\in Z}(f(a)-A\cdot d(a, x))-f(x)\ge 0, 
\]
and take 
$r\in (0, \infty)$
with 
$r>A+k$. 
Note that we have 
\[
\sup_{a\in Z}(f(a)-A\cdot d(a, x))\le B<\infty,
\] 
and this inequalities guarantee 
the existence of 
$k\in (0, \infty)$ 
taken above. 
We shall estimate 
$f(q)-rd(x, q)$. 
We divide the 
estimation 
into the case of 
$d(x, q)<\delta$
and the case of
$d(x, q)\ge \delta$. 
If 
$d(x, q)<\delta$, 
we have 
$f(q)-r\cdot d(q, x)\le 
f(q)\le   f(x)+\epsilon$. 
If $d(q, x)\ge \delta$, 
then 
we have 
\begin{align*}
&f(q)-rd(q, p)<
f(q)-(A+k)\cdot d(p, q)\\
&
=(f(q)-A\cdot d(p, q))-k\cdot d(p, q)\le 
(f(q)-A\cdot d(q, p))-\delta\cdot k\\
&
\le \sup_{a\in Z}(f(a)-A\cdot d(a, p))-
\left(\sup_{a\in Z}(f(a)-A\cdot d(a, p))-f(p)\right)=f(p).
\end{align*} 
Thus, 
 a sufficiently large number 
 $r\in (0, \infty)$, 
  we have
$f(x)\le \yowapp{f}{r}(x)\le f(x)+\epsilon$. 
This means that $\lim_{r\to \infty}\yowapp{f}{r}(x)=f(x)$. 
Using 
$\yomapp{f}{r}=-\yowapp{-f}{r}$, 
applying 
the previous discussion to 
$\yowapp{-f}{r}$ and $-f$, 
we also  obtain 
$\lim_{r\to \infty}\yomapp{f}{r}(x)=f(x)$. 
   Namely, 
 the map 
 $f$ 
 is a point-wise limit of 
  $\yomapp{f}{r}$ 
  and 
  $\yowapp{f}{r}$
   on 
  $Z$  
  as 
  $r\to \infty$.

Since 
$\yomapp{f}{r}(x)\le 
\yowapp{f}{r}(x)\le A\cdot \yocompdis{d}(x, p)+B$
for all 
$x\in \yocompsp{Z}{d}$, 
the functions 
  $\yomapp{f}{r}$ 
  and 
  $\yowapp{f}{r}$
are integrable 
with respect to 
$(Z, \mu)$, 
and hence, 
by the dominated convergence theorem
(or the monotone convergence theorem), 
we obtain 
\[
\lim_{r\to \infty}\int_{Z}\yomapp{f}{r}(x)\ d\mu(x)
=
\lim_{r\to \infty}\int_{Z}\yowapp{f}{r}(x)\ d\mu(x)
=
\int_{Z}f(x)\ d\mu(x). 
\]
In this setting,  take a sufficiently large 
$r\in (0, \infty)$ 
again, 
and 
put 
$M=\yomapp{f}{r}$ 
and 
$W=\yowapp{f}{r}$. 
Then 
we obtain 
the functions stated 
 in the proposition. 
 This finishes the proof. 
\end{proof}

Combining aforementioned statements, 
we can verify the 
following 
generalization of 
Theorem \ref{thm:wastop}, 
which gives   a description 
of the topologies 
 of the Wasserstein spaces 
 on incomplete spaces in a 
 similar 
 manner of the Portmanteau theorem.

\begin{thm}\label{thm:wastop20240918}
Let 
$(Z, d)$ 
be a metric space, 
and let 
$\{\mu_{i}\}_{i\in \zz_{\ge 0}}$
be a sequence 
in 
$\yowassp{Z}$
and 
take 
$\mu\in \yowassp{Z}$. 
Then the following statements are equivalent to 
each other. 
\begin{enumerate}[label=\textup{(\arabic*)}]

\item\label{item:wconvthm}
The sequence 
$\{\mu_{i}\}_{i\in \zz_{\ge 0}}$
converges to
 $\mu$
 in 
 $(\yowassp{Z}, \yowasdis{d})$.

\item\label{item:bddconvthm}
If a continuous function 
$f\colon Z\to \rr$
belongs to 
$\yodbddsp{Z}{d}$, 
then we have 
\[
\int_{Z}f(x)\ d\mu_{i}
\to \int_{Z}f(x)\ d\mu
\]
 as 
$i\to \infty$. 
\end{enumerate}
\end{thm}
%
\begin{proof}
We first assume that 
Statement \ref{item:wconvthm} is true.
To show 
Statement \ref{item:bddconvthm}, 
take a
continuous function
$f\colon Z\to \rr$
belongs to 
$\yodbddsp{Z}{d}$. 
Since 
$\yowasdis{d}(\mu_{i}, \mu)\to 0$
as 
$i\to \infty$,
Proposition \ref{prop:wawasamesame}
implies that 
\[
\yowasdis{\yocompdis{d}}(\yomext{\mu_{i}}, \yomext{\mu})\to 0
\]
as 
$i\to \infty$. 
Note that 
the set
$\yocompsp{Z}{d}\setminus Z$ 
is  null with respect to 
both 
$\yomext{\mu}$ 
and 
$\yomext{\mu_{i}}$
due to the constructions of 
$\yomext{\mu}$ 
and 
$\yomext{\mu_{i}}$
(Definition \ref{df:extmeasure}). 
For every 
$\epsilon \in(0, \infty)$, 
applying 
Proposition 
\ref{item:pointapproxlip}
to 
$f$ 
and 
$\yocompsp{Z}{d}$, 
for each 
$\epsilon \in (0, \infty)$, 
we obtain 
Lipschitz  functions 
$M$ 
and 
$W$ 
on
 $(\yocompsp{Z}{d}, \yocompdis{d})$
 such that 
 \[
 M\le f\le W, 
 \]
 \[
 \int_{Z}W-f\ d\mu<\epsilon, 
 \] 
 and 
 \[
 \int_{Z}f-M\ d\mu<\epsilon.
 \] 
Since 
$M, W\in \yodbddsp{\yocompsp{Z}{d}}{\yocompdis{d}}$, 
by 
$\yowasdis{\yocompdis{d}}(\yomext{\mu_{i}}, \yomext{\mu})\to 0$
as 
$i\to \infty$ 
and 
by 
Condition 
\ref{item:bddconv} 
in 
Theorem 
 \ref{thm:wastop}, 
we have  
\[
\int_{\yocompsp{Z}{d}}M\ d\yomext{\mu_{i}}\to \int_{\yocompsp{Z}{d}}M\ d\yomext{\mu}, 
\]
and 
\[
\int_{\yocompsp{Z}{d}}W\ d\yomext{\mu_{i}}\to \int_{\yocompsp{Z}{d}}W\ d\yomext{\mu}, 
\]
as 
$i\to \infty$. 
By 
$M\le f\le W$,
$\int_{Z}W-f\ d\mu<\epsilon$, 
and 
$\int_{Z}f-M\ d\mu<\epsilon$,  
we can compute as follows:
\begin{align*}
&
\int_{Z}f\ d\mu -\epsilon 
\le 
\int_{Z} M\ d\mu
=
\lim_{i\to \infty}\int_{Z} M\ d\mu_{i}
=
\liminf_{i\to \infty}\int_{Z} M\ d\mu_{i}
\\
&\le 
\liminf_{i\to \infty}\int_{Z} f\ d\mu_{i}
\le 
\limsup_{i\to \infty}\int_{Z} f\ d\mu_{i}
\le 
\limsup_{i\to \infty}\int_{Z} W\ d\mu_{i}\\
&
=
\lim_{i\to \infty}\int_{Z} W\ d\mu_{i}
=\int_{Z}W \ d\mu
\le \int_{Z}f \ d\mu+\epsilon. 
\end{align*}
In particular, 
we have 
\[
\int_{Z}f\ d\mu -\epsilon 
\le \liminf_{i\to \infty}\int_{Z} f\ d\mu_{i}
\le 
\limsup_{i\to \infty}\int_{Z} f\ d\mu_{i}
\le 
\int_{Z}f\ d\mu +\epsilon. 
\]
Since 
$\epsilon\in (0, \infty)$ 
is arbitrary, 
we conclude that
\[
\lim_{i\to \infty}\int_{Z}f\ d\mu_{i}
=\int_{Z}f\ d\mu. 
\]
This means that 
Statement
\ref{item:bddconvthm}
is valid. 

To verify the converse, 
 we assume that 
Statement
\ref{item:bddconvthm}
is 
true. 
Take 
an arbitrary  continuous function 
$F\colon \yocompsp{Z}{d}\to \rr$
belonging to 
$\yodbddsp{\yocompsp{Z}{d}}{\yocompdis{d}}$. 
Put $f=F|_{Z}$. 
Then 
$f\in \yodbddsp{Z}{d}$. 
Due to Statement 
\ref{item:bddconvthm}, 
we have 
\[
\lim_{i\to \infty}\int_{Z}f\ d\mu_{i}=
\int_{Z}f\ d\mu. 
\]
From this equality and  the fact that 
 and 
$\yocompsp{Z}{d}\setminus Z$
is null with respect to both 
each
$\yomext{\mu_{i}}$ 
and 
$\yomext{\mu}$, 
it follows that 
\[
\lim_{i\to \infty}\int_{\yocompsp{Z}{d}}F\ d\yomext{\mu_{i}}
=\int_{\yocompsp{Z}{d}}F\ d\yomext{\mu}. 
\]
Therefore, 
Condition 
\ref{item:bddconv}
of 
Theorem 
\ref{thm:wastop}
is true for 
$(\yocompsp{Z}{d}, \yocompdis{d})$, 
$\{\yomext{\mu_{i}}\}_{i\in \zz_{\ge 0}}$, 
and 
$\yomext{\mu}$. 
Thus, 
$\yowasdis{\yocompdis{d}}(\yomext{\mu_{i}}, \yomext{\mu})\to $ 
as 
$i\to \infty$, 
and hence, 
Proposition 
\ref{prop:wawasamesame}
proves  that 
$\yowasdis{d}(\mu_{i}, \mu)\to 0$
 as 
$i\to \infty$. 
This finishes the proof. 
\end{proof}

\begin{cor}\label{cor:wassametop}
Let 
$Z$ 
be a 
metrizable space, 
and fix
$\yoccl\in \yoslamet{Z}$. 
Then 
for every pair 
$d, e\in \yoccl$,
the metrics
$\yowasdis{d}$ 
and 
$\yowasdis{e}$
 generate
  the same topology on 
 $\yowassp{Z}$. 
\end{cor}
%
\begin{proof}
By the help of 
$\metdis_{X}(d, e)<\infty$, 
we conclude that 
$\yodbddsp{Z}{d}=\yodbddsp{Z}{e}$. 
Therefore 
Theorem 
\ref{thm:wastop20240918}
implies  Corollary 
\ref{cor:wassametop}. 
\end{proof}

Now we 
 observe that 
the relationship 
between 
the $1$-Wasserstein spaces and 
the supremum metrics on 
spaces of metrics. 

\begin{prop}\label{prop:wasisom}
Let 
$Z$ 
be a metrizable space, 
and fix a class 
$\yoccl\in \yoslamet{Z}$. 
Then 
the following statements are true:
\begin{enumerate}[label=\textup{(\arabic*)}]

\item\label{item:wasisom}
The map 
$\yowas \colon \met{Z}\to 
\met{\yowassp{Z}}$
defined by 
$d\mapsto \yowasdis{d}$
is an isometric embedding, 
i.e., 
for every pair 
$d, e\in \met{Z}$, 
we have 
\[
\metdis_{Z}(d, e)
=\metdis_{\yowassp{Z}}(\yowasdis{d}, \yowasdis{e}). 
\]

\item\label{item:wascomp}
If 
$d\in \yoccl$ 
is complete, 
then 
so is 
$\yowasdis{d}$. 

\end{enumerate}
\end{prop}
%
\begin{proof}
By  Corollary 
\ref{cor:oswasext}, 
we see that 
$\metdis_{Z}(d, e)\le 
\metdis_{\yowassp{Z}}(\yowasdis{d}, \yowasdis{e})$. 
Next, 
we prove the opposite inequality. 
Take 
$\alpha, \beta\in \yowassp{Z}$. 
For an arbitrary 
number 
$\epsilon\in (0, \infty)$, 
take 
$\nu\in \Pi(\alpha, \beta)$
such that 
\[
\int_{Z\times Z}e(s, t)\ d\nu\le \yowasdis{e}(\alpha, \beta)+\epsilon.
\]
In this setting, 
by the definition of the 
$1$-Wasserstein distance, 
we also obtain 
$\yowasdis{d}(\alpha,\beta)\le \int_{Z\times Z}d(s, t)\ d\nu$. 
Then we have 
\begin{align*}
&\yowasdis{d}(\alpha, \beta)-\yowasdis{e}(\alpha, \beta)
\le \int_{Z\times Z}d(s, t)\ d\nu-\int_{Z\times Z}e(s, t)\ d\nu+\epsilon \\
&=
\epsilon+\int_{Z\times Z}(d(s, t)-e(s, t)) \ d\nu
\le \epsilon+\int_{Z\times Z}|d(s, t)-e(s, t)| \ d\nu\\
&\le \epsilon+\int_{Z\times Z}\metdis_{Z}(d, e) \ d\nu
=
\epsilon+\metdis_{Z}(d, e). 
\end{align*}
Since 
$\alpha$,
 $\beta$, 
 and 
 $\epsilon$ 
are
 arbitrary, 
we obtain 
$\metdis_{\yowassp{Z}}(\yowasdis{d}, \yowasdis{e})
=\metdis_{Z}(d, e)$. 
Statement 
\ref{item:wascomp}
 follows from 
\cite[Theorem 2.7]{MR3858004}. 
\end{proof}



\subsubsection{Spaces of measurable functions}\label{subsec:spmea}
Fix  a topological space 
$\yomsp$, 
and a Borel probability measure 
$\yomea$ 
on 
$\yomsp$.
For a metrizable space 
$Z$, 
we denote by 
$\yoosmiscspw{Z}$
the set of all 
Borel measurable functions 
from
 $\yomsp$ 
 into 
 $Z$. 
 We define 
 $\yoosmiscemb{Z}\colon Z\to \yoosmiscspw{Z}$ defined 
 by 
$\yoosmiscemb{Z}(z)(a)=z$;
namely, 
each 
$\yoosmiscemb{Z}(z)$ 
 is a constant map. 
 For a metric 
 $d\in \met{Z}$, 
we also define a metric 
$\yoosmiscdis{d}$
on 
the space 
$\yoosmiscspw{Z}$
 by 
\[
\yoosmiscdis{d}(f, g)
=
\int_{\yomsp}d(f(t), g(t))\ d\yomea(t). 
\]
In general, 
the metric 
$\yoosmiscdis{d}(f, g)$
can take the value 
$\infty$. 
As is the case of 
the 
$1$-Wasserstein spaces, 
we define 
$\yoosmiscspc{Z}{\yoccl}$
by the set of all 
$f\in \yoosmiscspw{Z}$
such that 
$\yoosmiscdis{d}(f, \yoosmiscemb{Z}(a))<\infty$
for some/any 
$a\in Z$. 
When we fix 
a class 
$\yoccl\in \yoslamet{Z}$, 
we simply write 
$\yoosmiscsp{Z}=\yoosmiscspc{Z}{\yoccl}$. 

The space 
$(\yoosmiscsp{Z}, \yoosmiscdis{d})$
is an 
analogue of 
the ordinary 
$L^{1}$ 
spaces on 
the Euclidean spaces.
The author 
was
inspired by 
the papers
\cite{MR0368068}
and 
\cite{zarichnyi1996regular}
with respect to this construction. 

Let us observe some properties of 
spaces of measurable functions. 
\begin{prop}\label{prop:cbasics}
Let 
$\yomsp$
be a 
topological  space, and 
$\yomea$ 
be a Borel
probability measure on 
$\yomsp$. 
Let 
$Z$
be a metrizable space, 
fix 
$w\in\met{Z}$, 
and 
put 
$\yoccl=\yocclq{w}\in \yoslamet{Z}$. 
Then we have the following statements:
\begin{enumerate}[label=\textup{(\arabic*)}]
\item\label{item:oscext}
For every 
$d\in \yoccl$, 
and for every pair 
$x, y\in Z$, 
we have 
\[
\yoosmiscdis{d}\left(\yoosmiscemb{Z}(x), 
\yoosmiscemb{Z}(y)\right)
=d(x, y). 
\]
\item\label{item:cfinite}
For every 
$d\in \yoccl$, 
and 
for every pair  
$f,  g\in \yoosmiscspc{Z}{\yoccl}$, 
we have 
$\yoosmiscdis{d}(f, g)<\infty$. 
\item\label{item:cisom}
For every pair 
$d, e\in \yoccl$, 
we have
$\metdis_{Z}(d, e)=\metdis_{\yoosmiscsp{Z}}(\yoosmiscdis{d}, \yoosmiscdis{e})$. 
\end{enumerate}
\end{prop}
%
\begin{proof}
Statement 
\ref{item:oscext}
is follows from the definitions of 
$\yoosmiscdis{d}$ 
and 
$\yoosmiscemb{Z}$. 

Proof of 
Statement 
\ref{item:cfinite}: 
We only need to show the statement in the case 
of
 $g=\yoosmiscemb{Z}(a)$ 
 for some
 $a\in Z$. 
 By the definition of 
 $\yoosmiscspc{Z}{\yoccl}$, 
 we have 
 $\yoosmiscdis{w}(f, g)<\infty$. 
 Then 
we have 
\begin{align*}
\yoosmiscdis{d}(f, g)&=
\int_{\yomsp}d(f(t), a)\ d\yomea(t)
\\
&\le \int_{\yomsp}w(f(t), a)\ d\yomea(t)
+\metdis_{Z}(d, w)\\
&=
\yoosmiscdis{w}(f, g)
+\metdis_{Z}(d, w)
<\infty. 
\end{align*}
This proves Statement 
\ref{item:cfinite}.

Proof of \ref{item:cisom}: 
For every pair 
$f, g\in \yoosmiscsp{Z}$, 
we have 
\begin{align*}
&\yoosmiscdis{d}(f, g)-\yoosmiscdis{e}(f, g)
=\int_{\yomsp}d(f(t), g(t))-e(f(t), g(t))\ d\yomea(t)\\
&\le 
\int_{\yomsp}|d(f(t), g(t))-e(f(t), g(t))|\ d\yomea(t)
\le \int_{\yomsp}\metdis_{Z}(d, e)\ d\yomea(t)\\
&=
\metdis_{\yomsp}(d, e). 
\end{align*}
Then 
we obtain 
$\metdis_{\yoosmiscsp{Z}}(\yoosmiscdis{d}, \yoosmiscdis{e})\le \metdis_{Z}(d, e)$. 
Statement 
\ref{item:oscext}
shows that 
$\metdis_{Z}(d, e)\le \metdis_{\yoosmiscsp{Z}}(\yoosmiscdis{d}, \yoosmiscdis{e})$, 
and hence 
\[
\metdis_{Z}(d, e)=\metdis_{\yoosmiscsp{Z}}(\yoosmiscdis{d}, \yoosmiscdis{e}).
\]

This finishes the proof. 
\end{proof}

We next consider the 
topologies of 
$\yoosmiscsp{Z}$. 
We need the following observation on 
continuous maps on compact spaces. 
\begin{lem}\label{lem:cptetaep}
Let 
$Y$
 be a metrizable space, 
 and take 
 $\yoccl\in \yoslamet{Y}$
 and 
 $d, e\in \yoccl$. 
Let 
$K$ be a compact 
metrizable space, 
and 
$q\colon K\to Y$
be a continuous map. 
Then, 
for every 
$\epsilon \in (0, \infty)$, 
 there exists 
$\eta\in (0, \infty)$ such 
that
for every 
$t\in K$ and 
for every 
$y\in Z$, 
 if 
$d(q(t), y)<\eta$, 
then 
$e(q(t), y)<\epsilon$. 
\end{lem}
%
\begin{proof}
We employ 
a similar   idea as  the proof of 
\cite[Proposition 3]{MR4586584}. 
For each 
$s\in K$, 
we denote by 
$R(s)$
the set of 
all 
$u\in (0, \infty)$
such that 
for every $y\in Z$, 
if $d(q(s), y)<u$, 
then 
$e(q(s), y)+u<\epsilon$. 
Note that 
$R(s)\neq \emptyset$
since $e$ and $d$
generate the same topology on 
$X$. 
The definition of
 $R(s)$ 
 seems strange, but
 it is 
effective for our lemma. 
We now prove that 
the map 
$r\colon K\to (0, \infty)$
defined by 
$s\mapsto \sup R(s)$
is 
lower semi-continuous, 
i.e., 
for every 
$a\in \rr$, 
we shall show  that 
$r^{-1}((a, \infty))$
 is open in 
 $K$. 
According to the definition of 
$R(s)$, 
 we see that 
 $r(s)\le \epsilon$. 
 Take 
$s\in r^{-1}((a, \infty))$. 
By the 
definition of 
$r$, 
we can take 
$u\in R(s)$ 
such that 
$a<u-\delta_{0}<u\le r(s)$, 
where 
$\delta_{0}=2^{-1}(r(s)-a)\in (0, \infty)$. 
Fix  
$w\in \met{K}$. 
Since 
the map 
$q\colon K\to Y$
 is continuous, 
 we can find 
$\delta_{1}\in (0,\infty)$
such that 
\begin{enumerate}
\item 
if 
$w(s, t)<\delta_{1}$, 
then  we have 
$d(q(s), q(t))<\delta_{0}$
and $e(q(s), q(t))<\delta_{0}$. 

\end{enumerate}

Assume that 
$t\in U(s, \delta_{1}; w)$. 
Let us verify 
$u-\delta_{0}\in R(t)$. 
Namely, 
from now on, 
we shall show that, 
for every 
$y\in X$,  
the inequality 
$d(q(t), y)<u-\delta_{0}$
implies that 
$e(q(t), y)+(u-\delta_{0})<\epsilon$. 
Under the condition
$d(q(t), y)<u-\delta_{0}$, 
we have 
\[
d(q(s), y)\le 
d(q(s), q(t))+d(q(t), y)
<\delta_{0}+(u-\delta_{0})=u. 
\]
Namely, 
$d(q(s), y)<u$. 
Combining $d(q(s), y)<u$ and 
$u\in R(s)$,  we 
also have 
\begin{align*}
e(q(t), y)+(u-\delta_{0})
&\le 
e(q(t), q(s))+e(q(s), y)+(u-\delta_{0})
\\
&<\delta_{0}+e(q(s), y)+(u-\delta_{0})
= e(q(s), y)+u<\epsilon. 
\end{align*}
Thus 
$u-\delta_{0}\in R(t)$
and 
$a<r(t)$. 
This means that 
$U(s, \delta_{1}; w)\yosub  r^{-1}((a, \infty))$,
i.e., 
the set 
$ r^{-1}((a, \infty))$ is open in 
$K$. 
Thus 
$r$ 
is lower semi-continuous on 
$K$. 
Since $K$ is compact
and $r$ is lower semi-continuous, 
the map 
$r$ has a minimum $r_{\min}>0$. 
Put 
$\eta=\frac{1}{2}r_{\min}$. 
Then it   is the number  as desired. 
\end{proof}

We now show that  the topology induced by 
$\yoosmiscdis{d}$ does not  depend on 
the choice of 
$d\in \yoccl$
when 
$\yomsp$
 is Polish.  
\begin{prop}\label{prop:cbasicstop}
Let 
$\yomsp$
be a Polish space, and 
$\yomea$ 
be a Borel
probability measure on 
$\yomsp$. 
Let 
$Z$
be a metrizable space, 
fix 
$w\in\met{Z}$, 
put 
$\yoccl=\yocclq{w}\in \yoslamet{Z}$. 
Then 
for every pair 
$d, e\in \yoccl$, 
the metrics
$\yoosmisc(d)$ and 
$\yoosmisc(e)$ generate
 the 
same topology on 
$\yoosmiscsp{Z}$. 

\end{prop}
%
\begin{proof}
Fix 
$q\in \yoosmiscsp{Z}$. 
We will show that 
for every 
$\epsilon\in (0, \infty)$, 
there exists 
$\delta\in (0, \infty)$ 
such that 
if 
 $f\in  \yoosmiscsp{Z}$ 
 satisfies 
$\yoosmiscdis{d}(q, f)<\delta$, 
then  we 
have 
$\yoosmiscdis{e}(q, f)< \epsilon$.

Put 
$C=\metdis_{Z}(d, e)<\infty$ 
and 
take  a number 
$\eta_{0}\in (0, \infty)$ 
such that 
\begin{enumerate}[label=\textup{(k\arabic*)}, series=kproperties]
\item\label{item:kproperty}
$2(\eta_{0}+C)\eta_{0}< 2^{-1}\epsilon$.

\end{enumerate}

Since 
$q\colon \yomsp\to Z$ is 
Borel measurable and
 $\yomsp$ 
 is Polish, 
by 
\cite[Vol II, Lemma 6.10.16]{MR2267655}, 
the image 
$q(\yomsp)$
 of 
$q$
 is separable. 
 Note that 
 the map 
 $q\colon \yomsp\to q(\yomsp)$
 is also Borel measurable. 
 Thus 
we can apply 
 Lusin's theorem 
\cite[Theorem 17.12]{MR1321597}
to 
$q\colon \yomsp\to q(\yomsp)$
and 
to
the number 
$\eta_{0}$
 taken above. 
As a result, 
we obtain 
a
compact subset 
$K$ 
of 
$\yomsp$ 
such 
that 
\begin{enumerate}[resume*=kproperties]
\item\label{item:k123}
$\yomea(\yomsp\setminus K)< \eta_{0}$; 
\item\label{item:kconti}
the
restricted  function 
$q|_{K}\colon K\to Z$ is 
continuous on 
$K$. 
\end{enumerate}
Take
 a sufficiently small 
 number
 $\eta\in (0, \infty)$ 
 such that 
\begin{enumerate}[resume*=kproperties]
\item\label{item:k1}
we have 
$\eta<\eta_{0}$;

\item\label{item:k3}
for every  
$s\in K$, 
and for every 
$y\in Z$, 
if 
$d(q(s), y)<\eta$, 
then 
we have 
$e(q(s), y)<2^{-1}\epsilon$. 
\end{enumerate}
Remark that 
the property 
\ref{item:k3}
is guaranteed by 
\ref{item:kconti} and 
Lemma \ref{lem:cptetaep}.

Put 
$\delta=\eta^{2}$
and 
assume that 
$\yoosmisc(d)(q, f)<\delta(=\eta^{2})$. 
We also put 
\[
R=\{\, t\in \yomsp \mid \eta\le d(q(t), f(t))\, \}. 
\]
Then we have 
$\yomea(R)\le \eta$
by 
Chebyshev's inequality. 
Thus 
using 
\ref{item:kproperty}
and 
\ref{item:k1}, 
we also have 
\begin{align*}
&\int_{R} e(q(t), f(t))\ d\yomea(t)
\le 
\int_{R} d(q(t), f(t))+\metdis_{Z}(d, e)\ d\yomea(t)
\\
&=
\int_{R}d(q(t), f(t))\ d\yomea(t)
 +\int_{R}\metdis_{Z}(d, e)\ d\yomea(t)
  \\
&\le \yoosmisc(d)(q, f)+C\eta<\eta^{2}+C\eta=
(\eta+C)\eta\le 
(\eta_{0}+C)\eta_{0}. 
\end{align*}

For each 
$s\in (\yomsp\setminus R)\cap  K$, 
we have 
$d(q(s), f(s))<\eta$, 
and hence, 
by the property 
\ref{item:k3}
 of  
 $K$, 
we have 
$e(q(s), f(s))<2^{-1}\epsilon$. 
Thus, 
\[
\int_{(\yomsp\setminus R)\cap K} e(q(t), f(t))\ d\yomea(t) \le 
\int_{(\yomsp\setminus R)\cap K} 2^{-1}\epsilon
\ d\yomea(t) \le 
2^{-1}\epsilon. 
\]
Using 
$\metdis_{Z}(d, e)=C$
and 
$e(q(t), f(t))\le d(q(t), f(t))+C$
for all 
$t\in \yomsp$, 
 due to 
 and 
 \ref{item:kproperty}, 
\ref{item:k123}, 
and 
\ref{item:k1}
we obtain 
\begin{align*}
&\int_{(\yomsp\setminus R)\setminus K} e(q(t), f(t))\ d\yomea(t)= 
\int_{\yomsp\setminus (R\cup  K)}e(q(t), f(t))\ d\yomea(t)\\
&\le \int_{\yomsp\setminus (R\cup K)}d(q(t), f(t))+
C  \ d\yomea(t)\\
&\le \eta\cdot \yomea(\yomsp\setminus (R\cup K)) 
+C\cdot \yomea(\yomsp\setminus K)
\\
&\le 
\eta\cdot \yomea(\yomsp\setminus K) 
+C\cdot \yomea(\yomsp\setminus K)
<
\eta\cdot \eta_{0}+C\eta_{0}
< (\eta_{0}+C)\eta_{0}. 
\end{align*}

Combining these estimations, 
we also obtain 
\begin{align*}
&\int_{\yomsp} e(q(t), f(t))\ d\yomea(t)
=\int_{R} e(q(t), f(t))\ d\yomea(t)
+\int_{\yomsp\setminus R} e(q(t), f(t))\ d\yomea(t)\\
&=
\int_{R} e(q(t), f(t))\ d\yomea(t)+
\int_{(\yomsp\setminus R)\cap K} e(q(t), f(t))\ d\yomea(t)
+\\
&\int_{(\yomsp\setminus R)\setminus K} e(q(t), f(t))\ d\yomea(t)
\\
&< 
(\eta_{0}+C)\eta_{0}+2^{-1}\epsilon+(\eta_{0}+C)\eta_{0}= 2(\eta_{0}+C)\eta_{0}+2^{-1}\epsilon
\\
&<\frac{\epsilon}{2}+\frac{\epsilon}{2}= \epsilon. 
\end{align*}
That means that 
$\yoosmisc(e)(q, f)< \epsilon$. 
Replacing the role of 
$d$ 
with that of 
$e$, 
we conclude that 
$\yoosmiscdis{d}$ 
and 
$\yoosmiscdis{e}$
generate
 the 
same topology. 
\end{proof}

\begin{rmk}
A metric measure space
$(\yomsp, \yomea)$, 
where 
$\yomsp$ 
is Polish, 
is almost  the same as
the unit interval with the Lebesgue
 measure
 (see \cite[Lemma 4.2]{MR3445278} and 
 \cite[(17.41)]{MR1321597}). 
 Thus, our construction of 
 $\yoosmiscsp{Z}$ 
 roughly coincides with  the space of 
 measurable maps from 
 the unit  interval into
  $Z$. 
\end{rmk}

In what follows, 
whenever we use 
 $\yoosmiscsp{Z}$, 
we always assume that 
$\yomsp=\yoomega$; 
 namely, 
$\yomsp$ 
is the 
countable 
discrete space.
Under this assumption, 
the set of 
all measurable maps from 
$\yoomega$ 
into 
$Z$
is identical with  
the set of 
all maps from 
$\yoomega$
into 
$Z$. 
In the proof of 
the main theorem, 
we only  consider 
the measure 
$\yomasq$ 
on
$\yoomega$ 
defined by 
\[
\yomasq=\sum_{s\in \yoomega}\frac{1}{2^{s+1}}\yodmas{s}. 
\]

\subsection{Osmotic constructions}
In this subsection, 
we  introduce the notion of 
osmotic constructions as  
an abstraction 
from those three
constructions explained in the previous subsection.

Assume that 
a construction
$\yoosmisf$ 
of metric spaces
consists of 
three data
 for every metrizable space 
$Z$ 
and 
for every 
$\yoccl\in \yoslamet{Z}$:
\begin{enumerate}[label=\textup{(A\arabic*)}]
\item 
a metrizable space 
$\yoosmisfspq{Z}{\yoccl}$; 
\item 
a topological embedding 
$\yoosmisfemb{Z}\colon Z\to \yoosmisfspq{Z}{\yoccl}$; 
\item 
a map 
$\yoosmisf$
from 
$\yoccl$
into a set of metrics on 
$\yoosmisfspq{Z}{\yoccl}$. 
Namely,  
each 
$d\in \yoccl$ 
is mapped to 
a metric 
$\yoosmisfdis{d}$ 
on 
$\yoosmisfspq{Z}{\yoccl}$.  
\end{enumerate}
In this setting, 
we say that
$\yoosmisf$
is 
\yoemph{osmotic} 
if the following conditions are satisfied:
\begin{enumerate}[label=\textup{(O\arabic*)}]

\item 
for each 
$d\in \yoccl$
we have 
\[
\yoosmisfdis{d}
\left(\yoosmisfemb{Z}(x), \yoosmisfemb{Z}(y)\right)
=d(x, y); 
\]

\item 
for every pair 
$d, e\in \yoccl$, 
the metrics 
$\yoosmisfdis{d}$
and 
$\yoosmisfdis{e}$
generate the 
same topology of 
$\yoosmisfspq{Z}{\yoccl}$. 
Namely, 
we have 
$\yoosmisfdis{d}\in \met{\yoosmisfspq{Z}{\yoccl}}$. 
This means that 
the map 
$\yoosmisf$
from 
$\yoccl$
into a set of metrics on 
$\yoosmisfspq{Z}{\yoccl}$
becomes  a map 
$\yoccl\to \met{\yoosmisfspq{Z}{\yoccl}}$;

\item
the map 
$\yoosmisf\colon \yoccl\to \met{\yoosmisfspq{Z}{\yoccl}}$
satisfies 
that, 
for 
every pair 
$d, e\in \yoccl$, 
we have 
\[
\metdis_{Z}(d, e)=
\metdis_{\yoosmisfspq{Z}{\yoccl}}(\yoosmisfdis{d},\yoosmisfdis{e}). 
\] 

\end{enumerate}
Namely, 
these conditions  mean that 
$d$ 
can be naturally extended 
into a metric on 
$\yoosmisfsp{Z}$
as if  
 $d\in \yoccl$
 on 
 $Z$
osmoses 
from 
$\yoosmisfemb{Z}(Z)(=Z)$
into 
 an 
extension  space 
$\yoosmisfsp{Z}$.

We can summarize
the results in 
the previous subsection using the 
concept of 
osmotic constructions.
\begin{thm}\label{thm:osmotic}
The construction
$\yoosmisa$
 of 
 $\ell^{1}$-products, 
the construction 
$\yowas$ 
of 
$1$-Wasserstein
 spaces, 
and 
the construction
$\yoosmisc$ 
of spaces of measurable functions
are osmotic. 
\end{thm}
\begin{proof}
The theorem for 
the construction
$\yoosmisa$
 of 
 $\ell^{1}$-products
 follows from 
 Proposition 
 \ref{prop:ellisom}. 
 Combining 
Corollaries  \ref{cor:oswasext} and 
\ref{cor:wassametop}
and 
Proposition 
\ref{prop:wasisom}, 
we obtain the 
case of 
the construction 
$\yowas$ 
of 
$1$-Wasserstein
 spaces. 
Propositions 
\ref{prop:cbasics}
and 
\ref{prop:cbasicstop}
proves the theorem 
for the construction
$\yoosmisc$ 
of spaces of measurable functions. 
\end{proof}

\section{The Whitney--Dugundji  decomposition}\label{sec:wd}
To show our main result, 
we review the 
classical 
discussion called the 
Whitney--Dugundji 
decomposition, 
which is 
a special 
partition of unity  
on the complement of 
a closed subset in consideration. 
Such a decomposition in the Euclidean case 
was 
used by Whitney 
\cite{MR1501735}
to prove the famous Whitney extension theorem. 
The general case of metric spaces 
was found by 
Dugundji 
\cite{MR0044116}. 
See also 
\cite{ MR0049543}.

For the sake of 
convenience, 
in this paper, 
we
define such a decomposition 
as a quadruplet
of open covering, 
a partition of unity, 
two family of 
points.

\begin{df}\label{df:wd}
Let 
$X$ 
be a metrizable space, 
and 
$A$
be a closed subset of 
$X$.
Fix 
$w\in \met{X}$. 
For 
$k\in \zz_{\ge 0}$ 
and 
for a metric space 
$(X, w)$, 
we say that 
a 
quadruplet
\[
\yowdwd=
[
\yowdc,  
\{\yowdpe{O}\}_{O\in \yowdc},
\{\yowdae{O}\}_{O\in \yowdc},
\{\yowdge{O}\}_{O\in \yowdc}
]
\]
is 
a
\yoemph{$(k, w)$-WD collection}
if 
the following conditions are 
satisfied:
\begin{enumerate}[label=\textup{(WD\arabic*)}, series=wditems, leftmargin=*]
\item\label{item:wdbasic}
the family 
$\yowdc$
 is a locally finite 
 open covering of 
$X\setminus A$
consisting of open sets of 
$X\setminus A$, 
each 
$\yowdpe{O}$ 
is a point in 
$X\setminus A$, 
each 
$\yowdae{O}$ 
is a point in 
$A$, 
and 
each 
$\yowdge{O}$ 
is a continuous function
from 
$X$ 
to 
$[0, 1]$; 
\item\label{item:wdopenball}
For every 
 $O\in \yowdc$, 
we have 
\[
O\subset U\left(\yowdpe{O}, \frac{w(\yowdpe{O}, A)}{4^{k+1}}; w\right);
\] 
\item\label{item:wdapoint}
For every
 $O\in \yowdc$, 
we have 
\[
w(\yowdpe{O}, \yowdae{O})\le \frac{4^{k+1}+1}{4^{k+1}}w(\yowdpe{O}, A); 
\]
\item\label{item:wdsupp}
For every 
$O\in \yowdc$, 
we have 
\[
\supp (\yowdge{O})\yosub O, 
\]
where 
$\supp (\yowdge{O})$
stands for the support of 
$\yowdge{O}$ 
defined as 
the set 
$\{\, x\in X\mid \yowdge{O}(x)\neq 0\, \}$;

\item\label{item:wdpunity}
The family 
$\{\yowdge{O}\}_{O\in \yowdc}$
satisfies 
\[
\sum_{O\in \yowdc}\yowdge{O}(x)
=
\begin{cases}
0 & \text{if $x\in A$;}\\
1 & \text{if  $x\in X\setminus A$}. 
\end{cases}
\]
\end{enumerate}
In addition, 
if the quadruplet  satisfies the 
following condition, 
then 
it is called a 
\yoemph{strong $(k, w)$-WD collection}:
\begin{enumerate}[label=\textup{(WD\arabic*)}, resume*=wditems, leftmargin=*]
\item\label{item:wdstrong}
if 
$x\in X\setminus A$
 and 
$O\in \yowdc$ 
satisfy 
$\yowdge{O}(x)>0$, 
then 
\[
\yodiam_{w}(O)\le 16\cdot \frac{w(x, A)}{4^{k+1}}. 
\]
\end{enumerate}
\end{df}

Before proving the 
existence of 
strong 
WD collections, 
we provide
a convenient estimation 
of distances 
related to 
WD collections. 
\begin{lem}\label{lem:wd1}
Let 
$X$ 
be a metrizable space, 
and 
$A$ be a closed subset of 
$X$. 
Fix 
$w\in \met{X}$ 
and 
$k\in \zz_{\ge 0}$. 
Assume that 
a quadruplet
$\yowdwd=
[
\yowdc,  
\{\yowdpe{O}\}_{O\in \yowdc},
\{\yowdae{O}\}_{O\in \yowdc},
\{\yowdge{O}\}_{O\in \yowdc}
]$
is 
a 
$(k, w)$-WD collection. 
Then, 
for every 
$a\in A$ 
and 
for every 
$x\in X$, 
if 
$\yowdge{O}(x)>0$, 
then we have 
\[
w(a, \yowdae{O})\le 4 w(a, x). 
\]
\end{lem}
%
\begin{proof}
Our proof is based on 
that of 
\cite[Statement 2.3]{MR0049543}. 
Since 
$\yowdge{O}(x)>0$, 
the conditions 
\ref{item:wdsupp}
and 
\ref{item:wdopenball}
implies 
that 
$x\in O$
 and 
\[
O\subset U\left(\yowdpe{O}, \frac{w(\yowdpe{O}, A)}{4^{k+1}}\right). 
\] 
Thus we have
\[
w(x, \yowdpe{O})\le \frac{w(\yowdpe{O}, A)}{4^{k+1}}. 
\]
Using this estimation
and 
the condition 
\ref{item:wdapoint}, 
we also have 
\begin{align*}
w(a, \yowdae{O})
&\le 
w(a, x)+w(x, \yowdpe{O})+w(\yowdpe{O}, \yowdae{O})
\\
&\le
w(a, x)+\frac{1}{4^{k+1}}w(\yowdpe{O}, A)+
\frac{1+4^{k+1}}{4^{k+1}}w(\yowdpe{O}, A)\\
&=w(a, x)+\frac{4^{k+1}+2}{4^{k+1}}w(\yowdpe{O}, A). 
\end{align*}
Namely, 
\begin{align}
w(a, \yowdae{O})\le w(a, x)+\frac{4^{k+1}+2}{4^{k+1}}w(\yowdpe{O}, A).\label{al:wd1}
\end{align}
Moreover, 
by 
\[
w(\yowdpe{O}, A)\le 
w(\yowdpe{O}, a)\le 
w(\yowdpe{O}, x)+w(x, a)
\le \frac{1}{4^{k+1}}w(\yowdpe{O}, A)+w(x, a), 
\]
we obtain
\[
\left(1-\frac{1}{4^{k+1}}\right)w(\yowdpe{O}, A)
\le w(x, a), 
\]
and hence, we also obtain 
\begin{align}
w(\yowdpe{O}, A)\le \frac{4^{k+1}}{4^{k+1}-1}w(x, a).\label{al:wd2}
\end{align}

Therefore, 
combining these inequalities
\eqref{al:wd1}
and 
\eqref{al:wd2}, 
we confirm the following computations. 
\begin{align*}
&w(a, \yowdae{O})\le w(a, x)+
\frac{4^{k+1}+2}{4^{k+1}}w(\yowdpe{O}, A)
\le w(a, x)+\frac{4^{k+1}+2}{4^{k+1}-1}w(a, x)\\
&=\frac{2\cdot 4^{k+1}+1}{4^{k+1}-1}w(a, x)
=\frac{2+4^{-k-1}}{1-4^{-k-1}}w(x, a)
\le \frac{2+1}{3/4}w(x, a)
\le 4w(x, a). 
\end{align*}
This finishes the proof. 
\end{proof}

Now we prove 
the 
existence of 
$(k, w)$-WD 
collections. 
\begin{thm}\label{thm:swd}
Let 
$X$ 
be a metrizable space, 
and 
$A$ be a 
non-empty closed subset of 
$X$. 
Fix 
$w\in \met{X}$
and 
$k\in \zz_{\ge 0}$. 
Then there exists a 
strong 
$(k, w)$-WD 
collection. 
\end{thm}
%
\begin{proof}
For each 
$i\in \zz$, 
we define an open set 
$V_{i}$ 
 of 
$X\setminus A$
by 
\[
V_{i}=\{\,x\in X\setminus A\mid 2^{i-1}< w(x, A)<2^{i+1} \, \}. 
\]
Note that 
if 
$2\le |i-j|$, 
then 
$V_{i}\cap V_{j}=\emptyset$. 
Put 
\[
\mathcal{V}_{i}=\{\, V_{i}\cap U\left(x, w(x, A)4^{-k-1}; w\right)\mid x\in V_{i}\, \}
\]
Then 
each 
$\mathcal{V}_{i}$ 
is 
an
open covering of
 $V_{i}$. 
Put 
$\mathcal{W}=\bigcup_{i\in \zz_{\ge 0}}\mathcal{V}_{i}$. 
Since 
$X$ 
is paracompact (see
\cite[Corollary 1]{MR0026802} 
and 
\cite{MR0236876}), 
we can take a 
locally finite 
partition of unity 
$\{\psi_{a}\}_{a\in I}$
subordinated to 
$\mathcal{W}$
(see 
\cite[Proposition 2]{MR0056905}). 
By taking finite sums if necessary, 
we may assume that 
for every distinct pair  $a, b\in I$, 
 we have 
$\supp(\psi_{a})\neq \supp(\psi_{b})$. 
Put 
$\yowdc=\{\, \supp(\psi_{a})\mid a \in I\, \}$ 
and 
$\yowdge{O}=\psi_{a}$, 
where 
$\psi_{a}$ 
is a unique 
member  such that 
$\supp(\psi_{a})=O$. 
Then, 
the family 
$\yowdc$
 is a locally finite open covering of 
$X\setminus A$, 
and it refines 
$\mathcal{W}$. 

For each 
$O\in \yowdc$, 
choose 
$\yowdpe{O}\in X\setminus A$
so that 
\[
O\subset U(\yowdpe{O}, w(\yowdpe{O}, 4^{-k-1})).
\]
Note that 
if
 $O\yosub V_{i}$, 
then there exists 
$j\in \zz$ 
such that 
$|i-j|\le 1$ 
and 
$\yowdpe{O}\in V_{j}$. 
We also choose 
$\yowdae{O}\in A$
 so that 
we have 
\[
w(\yowdpe{O}, \yowdae{O})\le \frac{4^{k+1}+1}{4^{k+1}}w(\yowdpe{O}, A).  
\]

In this setting, 
if 
$\yowdge{O}(x)>0$
and 
$O\in \mathcal{W}_{i}$, 
then 
there exists 
$j\in \zz$
 with 
$|i-j|\le 1$
such that 
we have 
$x\in V_{j}\cap U(\yowdpe{O}, w(\yowdpe{O}, A)4^{-k-1}; w)$
and 
$\yowdpe{O}\in V_{j}$. 
By the construction, 
we observe that
the quadruplet
\[
[
\yowdc,  
\{\yowdpe{O}\}_{O\in \yowdc},
\{\yowdae{O}\}_{O\in \yowdc},
\{\yowdge{O}\}_{O\in \yowdc}
]
\]
is 
a
$(k, w)$-WD 
collection.

We next 
verify  that 
the quadruplet satisfies
the condition 
\ref{item:wdstrong}. 
Assume that 
$x\in X\setminus A$ 
and 
$O\in \yowdc$ 
satisfy 
$\yowdge{O}(x)>0$. 
Take 
$i\in \zz$
such that 
$O\in V_{i}$, 
and take 
$j\in \zz$
such that
$\yowdpe{O}\in V_{j}$. 
Then 
we have 
$|i-j|\le 1$. 
Under this setting, we obtain 
\begin{align*}
&\yodiam_{w}O\le 2\cdot 4^{-k-1}w(\yowdpe{O}, A)
\le 
2\cdot 4^{-k-1}\cdot 2^{j+1}
\le 2\cdot 4^{-k-1}\cdot 2^{i+2}\\
&
=2^{4}\cdot 4^{-k-1}2^{i-1}\le 2^{4}\cdot 4^{-k-1}w(x, A)
= 16\cdot w(x, A)/4^{k+1}
\end{align*}
This means that 
the quadruplet is 
a strong 
$(k, w)$-WD collection. 
\end{proof}



\section{Proof of the main result}\label{sec:proof}
The whole of this section is devoted to 
the proof of 
Theorem
 \ref{thm:main1}.

Throughout this section, 
let 
$X$ 
be a metrizable space, 
and 
$A$ 
be a closed subset of 
$X$.

First 
we fix 
$\yoccl\in \yoslamet{A}$, 
and 
take 
$m\in \met{A}$
such that 
$\yoccl=\yocclq{m}$. 
Fix 
$w\in \met{X}$
with 
$w|_{A^{2}}=m$. 
If 
$X$ 
is completely metrizable, 
we choose 
$m$ 
as a complete metric, 
which is guaranteed   by 
\cite[Theorem 1.4]{Ishiki2024smbaire}, 
and we also choose 
$w$
as 
a complete metric on 
$X$ 
(this is a variant of 
Hausdorf's metric extension theorem.  
see 
\cite{MR0024609}, 
 \cite{MR0230285}, 
 and 
 \cite{MR321026}).

For each 
$s\in \yoomega$, 
fix a
strong 
 $(s, w)$-WD collection
\[
\yowdwd_{s}=
[
\yowdce{s},  
\{\yowdpe{O, s}\}_{O\in \yowdce{s}},
\{\yowdae{O, s}\}_{O\in \yowdce{s}},
\{\yowdge{O, s}\}_{O\in \yowdce{s}}
]
\]
with respect to 
$X$ 
and 
$A$
(see Lemma \ref{lem:wd1}).

Using 
a partition of unity
(with respect to 
``$\sum$''), 
we shall 
construct a partition of unity 
with respect to 
``$\sup$''. 
\begin{lem}\label{lem:sigma}
For each 
$s\in \yoomega$, 
there exists a family 
$\{\yowdse{O, s}\}_{O\in \yowdce{s}}$
such that 
\begin{enumerate}[label=\textup{(\arabic*)}]

\item 
each 
$\yowdse{O, s}$
is a continuous function from 
$X$ 
to 
$[0, 1]$;

\item 
we have 
$\supp(\yowdse{O, s})=\supp(\yowdge{O, s})$
for all 
$O\in \yowdce{s}$; 

\item 
for every 
$x\in X$, 
there exists
 $O\in \yowdce{s}$
such that 
$\yowdse{O, s}(x)=1$. 
\end{enumerate}
\end{lem}
%
\begin{proof}
Define 
$\Phi\colon X\setminus A\to (0, 1]$
by 
$\Phi(x)=\sup_{O\in \yowdce{s}}\yowdge{O, s}(x)$. 
We also define 
$\yowdse{O, s}\colon X\to [0, 1]$
by 
\[
\yowdse{O, s}(x)=\frac{2}{\Phi(x)}\cdot \min
\left\{
\yowdge{O, s}(x), \frac{\Phi(x)}{2}
\right\}. 
\]
Then it is a map as required. 
\end{proof}

In what follows, 
we fix a family 
 $\{\yowdse{O, s}\}_{O\in \yowdce{s}}$
 stated in 
 Lemma 
 \ref{lem:sigma}.

Consider 
$\coprod_{s\in \yoomega}\yowdce{s}\times \{s\}$, 
and put 
$S=\yofmosp{\coprod_{s\in \yoomega}\yowdce{s}\times \{s\}}$
(see Subsection 
\ref{subsec:finsp}). 
Recall that 
$\yofmodisn$
 stands for the   
supremum metric on this space 
$S=\yofmosp{\coprod_{s\in \yoomega}\yowdce{s}\times \{s\}}$. 
Namely, 
the space 
$(S, \yofmodisn)$
 is a space of 
families 
$\{f(O, s)\}_{(O, s)\in \coprod_{s\in \yoomega}\yowdce{s}\times \{s\}}$
 of  real numbers 
indexed by 
the set 
$\coprod_{s\in \yoomega}\yowdce{s}\times \{s\}$. 
We denote by 
$0\in S$
the constant map taking the 
value 
$0\in \rr$.

For each 
$s\in \yoomega$,
fix  a continuous  map 
$\yocoeff{s}\colon [0, \infty)\to [0, 1]$ 
such that 
$\yocoeff{s}([0, 2^{-s}])=\{0\}$, 
$\yocoeff{s}((2^{-s}, 2^{-s+1}))\yosub (0, 1)$, 
and 
$\yocoeff{s}([2^{-s+1}, \infty))=\{1\}$. 
We also define
 $\yocoefff{s}\colon X\to [0, 1]$ by 
\[
\yocoefff{s}(x)=\yocoeff{s}(w(x, A)). 
\]
For 
$x\in X$, 
and
$s\in \yoomega$, 
we 
define 
a map 
$\yohsmapr(x, s)\colon \coprod_{s\in \yoomega}\yowdce{s}\times \{s\}\to \rr$
 by 
\[
\yohsmapr(x, s)(O, i)=
\begin{cases}
\yocoefff{s}(x)\cdot \yowdse{O, s}(x) & \text{if $s=i$;}\\ 
0 & \text{if $s\neq i$.}
\end{cases}
\]
Note that 
each 
$\yohsmapr(x, s)$
belongs to 
$S$. 
Let us see basic properties  of
 $\yohsmapr(x, s)$.

\begin{lem}\label{lem:hrmap}
The following statement are true:
\begin{enumerate}[label=\textup{(\arabic*)}]

\item\label{item:hrconti}
Fix $s\in \yoomega$. 
Then the map 
$\yohsmapr_{s}\colon X\to S$ 
defined by 
$x\mapsto \yohsmapr(x, s)$
is continuous. 

\item\label{item:hrle1}
If 
$x\in X$ 
and 
$N\in \yoomega$
satisfy
$\yofmodisn(\yohsmapr(x, N), 0)<1$, 
 then 
 we have 
 $w(x, A)\le 2^{-N+1}$. 
 
 \item\label{item:hrle12}
If 
$x, y\in X$ 
and 
$N\in \yoomega$
satisfy
$\yocoefff{N}(y)=\yocoefff{N}(y)=1$
and 
$\yofmodisn(\yohsmapr(x, N), \yohsmapr(y, N))<1$, 
 then 
there exists 
$O\in \yowdce{N}$ such that 
$\yowdse{O, N}(x)>0$
and 
$\yowdse{O, N}(y)>0$. 
In particular, 
we have 
$x, y\in O$, 
and we also obtain the inequality 
\[
w(x, y)\le 16\cdot 4^{-N-1}w(x, A).
\] 
 
 \end{enumerate}
\end{lem}
%
\begin{proof}
Statement 
\ref{item:hrconti}
follows from 
the fact that 
each 
$\yowdse{O, s}$ 
is continuous. 
We shall show 
\ref{item:hrle1}. 
If 
$x\in A$, 
we have 
$w(x, A)=0< 2^{-N+1}$. 
If 
$x\in X\setminus A$, 
then  the assumption
implies that
$\max_{O\in \yowdce{N}}\yocoefff{N}(x)\yowdse{O, N}(x)<1$. 
Take 
$P\in \yowdce{N}$
 with 
$\yowdse{P, N}(x)=1$
(see
 Lemma 
 \ref{lem:sigma}). 
Thus we have 
$\yocoefff{N}(x)<1$. 
By the definition of 
$\yocoefff{N}$, 
we 
conclude that 
$w(x, A)\le 2^{-N+1}$. 
Now we show 
\ref{item:hrle12}. 
In this setting, 
we see that 
$x, y\in A$
 and 
\[
\sup_{x\in X}|\yowdse{O, N}(x)-\yowdse{O, N}(y)|<1.
\] 
Take 
$O\in \yowdce{s}$
with 
$\yowdse{O, N}(x)=1$. 
Then 
$\yowdse{O, N}(y)>0$. 
Since 
$\supp(\yowdse{O, N})=\supp(\yowdge{O, N})$, 
the conditions
\ref{item:wdsupp} 
and  
\ref{item:wdstrong}
show
 \ref{item:hrle12}. 
\end{proof}

Next, 
we 
consider 
$(\yowassp{A}, \yowasdis{m})$
and 
we represent 
\[
\yoosmisdwsp{A}=
\yoosmisasp{\yowassp{A}}
=\yoosmisasp{\yowassp{A}, S}=
\yowassp{A}\times S, 
\] 
and 
\[
\yoosmisdwdis{d}
=\yoosmisadis{\yowasdis{d}}
=\yowasdis{d}\times_{\ell^{1}} \yofmodisn. 
\]
for every 
$d\in \yoccl$. 
In the proof of 
the main theorem, 
we only  consider 
the measure 
$\yomasq$
 on
$\yoomega$ 
defined by 
\[
\yomasq=\sum_{s\in \yoomega}\frac{1}{2^{s+1}}\yodmas{s}. 
\]
Using 
the  measure space 
$(\yoomega, \yomasq)$, 
we 
construct 
the space
$\yoosmiscsp{\yoosmisdwsp{A}}$
of measurable functions
(see Subsection 
\ref{subsec:spmea}). 
Employing  the zero element 
$0\in S$, 
we also construct a topological embedding 
$\yoosmisaemb{\yowassp{A}}$
(see Subsection 
\ref{subsec:ell1}). 
Then 
we can obtain  an topological embedding 
$I\colon A\to \yoosmiscsp{\yoosmisdwsp{A}}$
as a composition 
$\yoosmiscemb{\yoosmisdwsp{A}}\circ\yoosmisaemb{\yowassp{A}}\circ \yowasemb{A}$. 
Namely, 
$I(a)$ is a constant  map 
$\yoomega\to \yoosmisdwsp{A}$
such that 
$I(a)(s)=(\yodmas{a}, 0)$, 
where 
$0\in S$. 
Now, using 
$w$, 
we will construct a 
topological embedding 
$\yosubmap\colon X\to  \yoosmiscsp{\yoosmisdwsp{A}}$
such that 
$\yosubmap|_{A}=I$.

For 
$x\in X$, 
and for 
$s\in \yoomega$
we define 
$\yohsmapl(x, s)\in \yowassp{A}$ 
by 
\[
\yohsmapl(x, s)=
\begin{cases}
\sum_{O\in \yowdce{s}}\yowdge{O, s}(x)
\cdot \yodmas{a_{O, s}} & \text{if $x\in X\setminus A$}\\
\yodmas{x} & \text{if $x\in A$}
\end{cases}
\]

\begin{lem}\label{lem:hlmap}
Fix 
$s\in \yoomega$. 
Then
the  map 
$\yohsmapl_{s}\colon X\to \yowassp{A}$
defined by 
$x\mapsto \yohsmapl(x, s)$ 
is continuous. 
\end{lem}
%
\begin{proof}
The lemma 
follows from 
Lemma 
\ref{lem:wasfinfin}
and the local finiteness of 
$\{\supp(\yowdge{O, s})\}_{O\in \yowdce{s}}$. 
\end{proof}

For 
$x\in X$, 
 we put 
$\yosubmap_{s}(x)=(\yohsmapl(x, s), \yohsmapr(x, s))$, 
and 
define 
$\yosubmap(x)\colon \yoomega\to 
\yoosmisdwsp{A}=\yowassp{A}\times S$
by 
$\yosubmap(x)(s)=
\yosubmap_{s}(x)$.
Then 
$\yosubmap$
is a map from 
$X$
into 
$\yoosmiscspw{\yoosmisdwsp{A}}$.

In the next lemma, we shall see that 
$\yosubmap$
 is actually a map into 
 the space
$\yoosmiscsp{\yoosmisdwsp{A}}$.

\begin{lem}\label{lem:jbounded}
For every 
$x\in X$, 
we have 
$\yosubmap(x)\in \yoosmiscsp{\yoosmisdwsp{A}}$. 
Namely, 
for every 
$x\in X$
and for every 
$y\in A$, 
we have
\[
\yoosmiscdis{\yoosmisdwdis{m}}(\yosubmap(x), \yosubmap(y))\le 4w(x, y)+1<\infty.
\]
In particular, if 
$w$ 
is bounded, 
then so is 
$\yoosmiscdis{\yoosmisdwdis{m}}$. 
\end{lem}
%
\begin{proof}
For every 
$s\in \yoomega$, 
Lemmas
\ref{lem:wasfinfin}
and 
\ref{lem:wd1}
show
 that
\begin{align*}
&\yoosmisdwdis{m}(\yosubmap_{s}(x), \yosubmap_{s}(y))
\le \sum_{O\in \yowdce{s}}\yowdge{O}(x)\cdot 
m(\yowdae{O, s}, y) +1\\
&
\le
\sum_{O\in \yowdce{s}}\yowdge{O}(x)\cdot 
4w(x, y)
+1
=
4w(x, y)\cdot \sum_{O\in \yowdce{s}}\yowdge{O}(x)\cdot 
+1\\
&
=
 4w(x, y)+1. 
\end{align*}
Thus, 
we have 
\begin{align*}
&\yoosmiscdis{\yoosmisdwdis{d}}
(\yosubmap(x), \yosubmap(y))
\le \int_{\yoomega}4w(x, y)+1\ d\yomasq
= 4w(x, y)+1<\infty. 
\end{align*}
This completes the proof. 
\end{proof}

By the help of 
Lemma 
\ref{lem:jbounded}, 
in what follows, 
we often represent 
\[
\yoosmislldwsp{A}
=\yoosmiscsp{\yoosmisdwsp{A}}, 
\]
and 
\[
\yoosmislldwdis{d}
=\yoosmiscdis{\yoosmisdwdis{d}}, 
\]
for $d\in \yoccl$.

\begin{prop}\label{prop:jextension}
For every 
$d\in \yoccl$, 
and for 
every pair 
$x, y\in A$, 
we have 
$\yosubmap|_{A}=I$. 
Moreover, we also have 
\[
\yoosmislldwdis{d}
(\yosubmap(x), \yosubmap(y))=d(x, y)
\]
\end{prop}
%
\begin{proof}
The proposition 
is deduced 
from
the fact that 
these constructions are 
osmotic
(Theorem 
\ref{thm:osmotic}). 
\end{proof}

Next,
 let us prove that 
$\yosubmap\colon X\to \yoosmislldwsp{A}$ 
is a topological embedding.

\begin{lem}\label{lem:jconti}
The map 
$\yosubmap\colon X\to \yoosmislldwsp{A}$
is continuous. 
\end{lem}
%
\begin{proof}
Take 
an arbitrary  point 
$p\in X$
and 
$\epsilon\in (0, \infty)$. 
Take a sufficiently large number 
$N$ such that 
$2^{-N}\le \alpha \epsilon$, 
where 
$\alpha \in (0, \infty)$ 
satisfies that 
\begin{align}
&\alpha \le\min\left\{\frac{1}{2}\cdot \frac{1}{20w(p, A)+1},\frac{1}{4}\right\}.\tag{$\alpha$-1}\label{al:aa}
\end{align}
We divide 
the proof into two cases.

\yocase{1}{$p\in X\setminus A$}
Take 
$b\in A$ 
such that 
$w(p, b)\le 2w(p, A)$. 
We take a
sufficiently small number 
$\eta\in (0, \infty)$
 so that 
the next statements are true:
\begin{enumerate}[label=\textup{(a\arabic*)}]

\item 
We have 
$\eta\le w(p, A)$,

\item\label{item:aeta1}
For every 
$z\in U(p, \eta; w)$, 
we have 
$w(z, A)\le w(p, A)+\eta$. 
As a result, we observe that 
$w(b, z)
\le w(b, p)+w(p, z)
\le 2w(p, A)+\eta
\le 3w(p, A)$.

\item\label{item:aeta2}
For every 
$z\in U(p, \eta; w)$,
 and 
for every 
$s\in \{0, \dots, N\}$, 
we have 
\[
\yofmodisn(\yohsmapr(z, s), \yohsmapr(p, s))\le 
\frac{\epsilon}{4}.
\]

\item\label{item:aeta3}
For every 
$z\in U(p, \eta; w)$, 
and 
for every 
$s\in \{0, \dots, N\}$, 
we have 
\[
\sum_{O\in \yowdce{s}}
\left|\yowdge{O, s}(z)-\yowdge{O, s}(p)\right|\le 
\frac{1}{8w(p, A)+1}\cdot \frac{\epsilon}{4}. 
\]
\end{enumerate}
First, 
we provide 
an 
upper estimation of 
$\yoosmisdwdis{m}(\yosubmap_{s}(z), \yosubmap_{s}(p))$
for every 
$s\in \{0, \dots, N\}$. 
For every 
$s\in \{0, \dots, N\}$, 
Lemma 
\ref{lem:wd1} 
implies that 
\[
m(\yowdae{O, s}, b)
\le 4w(p, b)\le 8 w(p, A). 
\]
Then, 
for every 
$s\in \{0, \dots, N\}$, 
Lemma 
\ref{lem:wasfinfin}
and 
\ref{item:aeta3} 
show  that 
\begin{align*}
&\yowasdis{w}(\yohsmapl(z, s), \yohsmapl(p, s))
\le 
\sum_{O\in \yowdce{s}}
|\yowdge{O, s}(z)-\yowdge{O, s}(p)|m(b, \yowdae{O, s})
\\
&\le 
8w(p, A)\cdot \sum_{O\in \yowdce{s}}
|\yowdge{O, s}(z)-\yowdge{O, s}(p)| \\
&\le 
\frac{1}{8w(p, A)+1}\cdot \frac{\epsilon}{4}\cdot  
8w(p, A)\le 
\frac{\epsilon}{4}. 
\end{align*}
Next, 
we estimate 
$\yoosmisdwdis{m}(\yosubmap_{s}(z), \yosubmap_{s}(p))$. 
By the argument discussed above, 
and by 
\ref{item:aeta2}, 
for every 
$s\in \{0, \dots, N\}$, 
we have 
\begin{align*}
&\yoosmisdwdis{m}
(\yosubmap_{s}(z), \yosubmap_{s}(p))\\
&=\yowasdis{m}(\yohsmapl(x, s), \yohsmapl(p, s))+
\yofmodisn(\yohsmapr(x, s), \yohsmapr(p, s))
\le 
\frac{\epsilon}{4} +\frac{\epsilon}{4}
=\frac{\epsilon}{2}. 
\end{align*}

From now on, 
we estimate 
$\yoosmisdwdis{m}(\yosubmap_{s}(z), \yosubmap_{s}(p))$ 
for 
$s\in \yoomega$
 with 
$N\le s$. 
For every 
$s\in \yoomega$ 
with 
$s\ge N$,
using 
Lemma 
\ref{lem:wasfinfin} 
again, 
and
using 
Lemma
 \ref{item:aeta1}, 
 we have 
\begin{align*}
&\yowasdis{m}(\yohsmapl(z, s), \yohsmapl(p, s))\le 
\sum_{O\in \yowdce{s}}
\left|\yowdge{O, s}(z)-\yowdge{O, s}(p)
\right|w(\yowdae{O, s}, b)\\
&\le 
\sum_{O\in \yowdce{s}}
\left(
\yowdge{O, s}(z)w(\yowdae{O, s}, b)+
\yowdge{O, s}(p)
w(\yowdae{O, s}, b)\right)=\\
&
\sum_{O\in \yowdce{s}, \yowdge{O, s}(z)>0}
\yowdge{O, s}(z)w(\yowdae{O, s}, b)+
\sum_{O\in \yowdce{s}, \yowdge{O, s}(p)>0}
\yowdge{O, s}(z)w(\yowdae{O, s}, b)\\
&\le 
4w(b, z)\cdot \sum_{O\in \yowdce{s}, \yowdge{O, s}(z)>0}
\yowdge{O, s}(z)
+
4w(b, p)\cdot \sum_{O\in \yowdce{s}, \yowdge{O, s}(p)>0}
\yowdge{O, s}(z)
\\
&=
4w(b, z)+4w(b, p)
\le4\cdot 3w(p, A)+4\cdot 2w(p, A)
\le 20w(p, A). 
\end{align*}
Since  
$\yofmodisn(\yohsmapr(z, s), \yohsmapr(p, s))\le 1$
is always true, 
for every 
$s\in \yoomega$ 
with 
$N\le s$, 
we also have 
\begin{align*}
&\yoosmisdwdis{m}
(\yosubmap_{s}(z), \yosubmap_{s}(p))
\le
\yowasdis{m}(\yohsmapl(z, s), \yohsmapl(p, s))+1\\
&
\le 20w(x, A)+1. 
\end{align*}

Therefore, 
combining the estimations obtained above, 
we see: 
\begin{align*}
&\yoosmislldwdis{m}
(\yosubmap(z), \yosubmap(p))
=
\int_{\yoomega}
\yoosmisdwdis{m}
(\yosubmap_{s}(z), \yosubmap_{s}(p))
\ d\yomasq(s)\\
&=
\sum_{s=0}^{\infty}\frac{1}{2^{s+1}}
\yoosmisdwdis{m}
(\yosubmap_{s}(z), \yosubmap_{s}(p))\\
&
\le
\sum_{s=N}^{\infty}
\frac{\yoosmisdwdis{m}
(\yosubmap_{s}(z), \yosubmap_{s}(p))
}{2^{s+1}}
+\sum_{s=0}^{N}\frac{\yoosmisdwdis{m}
(\yosubmap_{s}(z), \yosubmap_{s}(p))
}{2^{s+1}}\\
 &\le 
 \sum_{s=N}^{\infty}\frac{20w(x, A)+1}{2^{s+1}}
+\sum_{s=0}^{N}\frac{2^{-1}\epsilon}{2^{s+1}}\le 
\frac{2(20w(x, A)+1)}{2^{N+1}}+\frac{\epsilon}{2}\\
&\le 
(20w(x, A)+1)2^{-N}+\frac{\epsilon}{2}
\le \alpha(20w(x, A)+1)\epsilon+\frac{\epsilon}{2}
\\
&\le \frac{\epsilon}{2}+\frac{\epsilon}{2}
=\epsilon.
\end{align*}
This means 
that 
$J$ 
is continuous at 
$p$.

\yocase{2}{$p\in A$}
Take a sufficiently
 small 
number 
$\eta\in (0, \infty)$
so that 
the next 
statements are satisfied: 
\begin{enumerate}[label=\textup{(b\arabic*)}]

\item\label{item:b1}
We have $\eta\le 2$.

\item\label{item:b2}
For  every 
 point 
 $z\in U(p, \eta; w)$, 
we have 
$w(z, A)\le 4^{-1}\cdot 2^{-N}$. 

\item\label{item:b3}
For  every 
 point 
 $z\in U(p, \eta; w)$, 
 and 
for 
every
$s\in \{0, \dots, N\}$, 
we
have
\[
\yofmodisn(\yohsmapr(z, s), \yohsmapr(p, s))\le 4^{-1}\epsilon. 
\]
\end{enumerate}

Under this conditions, we obtain:
\begin{enumerate}[label=\textup{(c\arabic*)}]

\item\label{item:b4}
For  every 
 point 
 $z\in U(p, \eta; w)$, 
 and 
for every 
$s\in \{0, \dots, N\}$,
we also have 
\begin{align*}
& 
\yowasdis{m}(\yohsmapl(z, s), \yohsmapl(p, s))=
\yowasdis{m}(\yohsmapl(z, s), \yodmas{p})\\
&=
\sum_{O\in \yowdce{s}}
\yowdge{O, s}(z)\cdot w(\yowdae{O, s}, p)
\le 4w(z, A)\le 2^{-N}. 
\end{align*}
\item\label{item:b5}
For  every 
 point 
 $z\in U(p, \eta; w)$, 
if  
$z\not\in A$, 
then 
for every 
$s\in \zz_{\ge 0}$
we have 
\begin{align*}
&\yoosmisdwdis{m}
(\yosubmap_{s}(z), \yosubmap_{s}(p))\le \yowasdis{m}(\yohsmapl(x, s), \yohsmapl(p, s))+1\\
& = 
 \yowasdis{m}(\yohsmapl(x, s), \yodmas{p})+1
 =
 \sum_{O\in \yowdce{s}}
\yowdge{O, s}\cdot
 w(\yowdae{O, i}, p)+1\\
 &
\le 4w(z, A)+1
\le 2^{-N}+1
\le 2. 
\end{align*}

\item\label{item:b6}
For  every 
 point 
 $z\in U(p, \eta; w)$, 
if 
 $z\in A$, 
 then we have 
\[
w(z, p)=m(z, p)\le \eta\le 2.
\]
\end{enumerate}
Then, 
due to 
\ref{item:b3}
and 
\ref{item:b4}, 
for every 
$s\in \{0, \dots, N\}$, 
we have 
\begin{align*}
&\yoosmisdwdis{m}(\yosubmap_{s}(z), \yosubmap_{s}(p))\\
&=
 \yowasdis{m}(\yohsmapl(z, s), \yohsmapl(p, s))+
 \yofmodisn(\yohsmapr(z, s), \yohsmapr(p, s))\\
 &
<2^{-N}+4^{-1}\epsilon
\le (\alpha+4^{-1})\epsilon. 
\end{align*}
According to 
\ref{item:b5}
and 
\ref{item:b6}, 
for 
$s\in \yoomega$ 
with 
$N\le s$, 
we also  have 
\[
\yoosmisdwdis{m}(\yosubmap_{s}(z), \yosubmap_{s}(p))\le 2. 
\]
Thus 
we obtain  
\begin{align*}
&\yoosmislldwdis{m}
(\yosubmap(z),\yosubmap(p))
=
\int_{\yoomega}\yoosmisdwdis{m}
(\yosubmap_{s}(z), \yosubmap_{s}(p))
\ d\yomasq(s)\\
&=
\sum_{s=0}^{\infty}2^{-(s+1)}
\yoosmisdwdis{m}
(\yosubmap_{s}(z), \yosubmap_{s}(p))
\le 
\sum_{s=N}^{\infty}\frac{2}{2^{s+1}}
+
\sum_{n=0}^{N}\frac{(\alpha+4^{-1})\epsilon}{2^{s+1}}\\
&\le 
2\cdot 2^{-N}+(\alpha+4^{-1})\epsilon 
\le 2\alpha \epsilon+(\alpha+4^{-1})\epsilon 
\le \frac{\epsilon}{2}+\frac{\epsilon}{2}
=\epsilon. 
\end{align*}
This means that 
$\yosubmap$ 
is continuous at 
$p\in A$. 

In any case, 
we confirm that 
$\yosubmap$
 is continuous. 
Hence the proof is completed. 
\end{proof}

\begin{prop}\label{prop:jhomeo}
The map 
$\yosubmap\colon X\to \yoosmislldwsp{A}$
is homeomorphic. 
\end{prop}
%
\begin{proof}
Take
an  arbitrary point 
 $p\in X$,
 and 
 an arbitrary number
$\epsilon \in (0, \infty)$. 
To prove the proposition, 
 we will find a 
 sufficiently small 
$\eta\in (0, \infty)$
 so  that 
if we have 
$\yoosmislldwdis{m}
(\yosubmap(p), \yosubmap(z))<\eta$, 
then 
$w(z, p)\le \epsilon$. 
Put 
$V_{\eta}=\{\, z\in X\mid \yoosmislldwdis{m}
(\yosubmap(p), \yosubmap(z))<\eta\, \}$. 

\yocase{1}{$p\in X\setminus A$}
Take a sufficiently large 
$N\in \yoomega$ 
so that 
$2^{-N+1}<w(p, A)$
and 
$16\cdot 4^{-N-1}w(p, A)\le \epsilon$. 
We also take a 
sufficiently small number 
$\eta$ 
so that 
\begin{enumerate}[label=\textup{(a\arabic*)}]
\item\label{item:a120240920}
$2^{N}\eta<1$;
\item\label{item:202499a2}
$2^{-N+1}< w(z, A)$. 
\end{enumerate}
The condition 
\ref{item:202499a2}
implies that 
every
 $z\in V_{\eta}$ 
 satisfies 
 that 
$\yocoefff{N}(z)=\yocoefff{N}(p)=1$. 
Then, 
by the definition of 
the metric
 $\yoosmislldwdis{m}$
 and by 
 \ref{item:a120240920}, 
for every 
$z\in V_{\eta}$, 
we have 
\[
\yofmodisn(\yohsmapr(z, N), \yohsmapr(p, N))
\le 2^{N}\eta< 1. 
\]
Thus, Statement \ref{item:hrle12}
in Lemma \ref{lem:hrmap} implies that 
\[
w(z, p)\le 16\cdot 4^{-N-1}w(p, A)\le \epsilon.
\] 
This is the inequality that we want to prove.

\yocase{2}{$p\in A$}
Take a sufficiently  large number 
$N\in \zz_{\ge 0}$ 
so that 
$68\cdot 2^{-N}\le 2^{-1} \epsilon$.  
We also take 
a sufficiently small
number 
$\eta\in (0, \infty)$
 so that 
$2^{N}\eta<1$
and 
$\eta<2^{-1}\epsilon$. 
If 
$z\in V_{\eta}$ belongs to 
$A$, 
then 
the inequality 
$\yoosmislldwdis{m}
(\yosubmap(p), \yosubmap(z))<\eta$ 
implies that 
$w(z, p)\le \eta$
since 
we already know 
$\yoosmislldwdis{m}(\yosubmap(x), \yosubmap(y))=m(x, y)$ 
whenever 
$x, y\in A$
 (see Proposition  
 \ref{prop:jextension}). 
Thus, 
we may assume that 
$z\in X\setminus A$. 
In this setting, 
by the definition of 
the metric
 $\yoosmislldwdis{m}$, 
the point 
$z\in V_{\eta}$ 
satisfies 
\[
\yofmodisn(\yohsmapr(z, N), \yohsmapr(p, N))=
\yofmodisn(\yohsmapr(z, N), 0)
\le 2^{N}\eta <1. 
\]
Hence 
Statement 
\ref{item:hrle1}
in Lemma
 \ref{lem:hrmap}
implies that 
$w(z, A)\le 2^{-N+1}$. 
Put 
$\yoidx=\{\, O\in \yowdce{N}\mid \yowdge{O, N}(z)>0\, \}$. 
We obtain 
\[
\yowasdis{m}(\yohsmapl(p, N), \yohsmapl(z, N))
=\sum_{O\in \yoidx}\yowdge{O, N}(z)\cdot 
w(\yowdae{O, N}, p)
\]
Take
 $P\in\yoidx$ 
 such that 
$w(\yowdae{P, N}, p)=\min_{O\in \yoidx} w(\yowdae{O, N}, p)$. 
Thus we have 
\begin{align*}
&\yowasdis{m}(\yohsmapl(p, N), \yohsmapl(z, N))
=\sum_{O\in \yoidx}\yowdge{O, N}(z)\cdot 
w(\yowdae{O, N}, p)
\\
&
\ge 
\sum_{O\in \yoidx}\yowdge{O, N}(z)\cdot 
w(\yowdae{P, N}, p)
=
w(\yowdae{P, N}, p)\cdot 
\sum_{O\in \yoidx}\yowdge{O, N}(z)
=w(\yowdae{P, N}, p), 
\end{align*}
and then 
this estimation indicates 
\[
w(\yowdae{P, N}, p)\le 
\yoosmisdwdis{m}
(\yosubmap_{N}(p), \yosubmap_{N}(z))\le 
\yoosmislldwdis{m}
(\yosubmap_{N}(p), \yosubmap_{N}(z)). 
\]
Thus 
$w(\yowdae{P, N}, p)\le \eta$. 
Since 
$\yowdge{O, N}(z)>0$
 for all 
$O\in \yoidx$, 
using 
the condition 
\ref{item:wdstrong}, 
we 
also have 
\[
w(z, \yowdpe{P, N})\le 16\cdot  4^{-N-1}w(z, A)
\le 16\cdot w(z, A)\le 32\cdot 2^{-N}.
\] 
The 
condition 
\ref{item:wdapoint}
 also 
implies that 
\begin{align*}
&w(\yowdpe{P, N}, \yowdae{P, N})\le 
(1+4^{-N-1})w(\yowdpe{P, N}, A)\\
&\le 
(1+4^{-N-1})(w(\yowdpe{P, N}, z)+w(z, A))\\
&\le 
(1+4^{-N-1})(1+8\cdot 4^{-N-1}) w(z, A)\le 2\cdot 9w(z, A)\\
&\le 18 w(z, A)\le 
36 \cdot 2^{-N}.
\end{align*}
As a result,
 we see the following estimation:
\begin{align*}
w(p, z)&\le w(p, \yowdae{P, N})
+w(\yowdae{P, N}, \yowdpe{P, N})+w(\yowdpe{P, N}, z)\\
&
\le \eta+ 36\cdot  2^{-N}+32\cdot 2^{-N}
=\eta+68\cdot 2^{-N}
\le  2^{-1}\epsilon+2^{-1}\epsilon=\epsilon. 
\end{align*}

In any case, 
we
 have 
 $w(p, z)\le \epsilon$. 
 Since 
 $p$ 
 and 
 $\epsilon$
 are arbitrary, 
 we conclude that 
 $\yosubmap$
 is homeomorphic. 
 This finishes the proof. 
\end{proof}

We next consider 
the topologies on 
$\yosubmap(X)$ 
induced by 
$\yoosmislldwdis{d}$. 
\begin{lem}\label{lem:jjbasics}
The following statements are true: 
\begin{enumerate}[label=\textup{(\arabic*)}]

\item\label{item:jjbounded}
For every 
$d\in \yoccl$ 
and 
for every pair 
$x, y\in X$, we have
\[
\yoosmislldwdis{d}
(\yosubmap(x), \yosubmap(y))<\infty.
\]

\item\label{item:jjtop}
 For every 
$d\in \yoccl$,  
the metric 
$\yoosmislldwdis{d}$
generates  the same topology on 
$\yosubmap(X)$ 
as 
$\yoosmislldwdis{w}$. 

\end{enumerate}
\end{lem}
%
\begin{proof}
Statement 
\ref{item:jjbounded}
follows 
 from 
 Lemma 
 \ref{lem:jbounded}
 and 
Statements
 \ref{item:cfinite} 
in 
Proposition 
 \ref{prop:cbasics}. 
 
Statement 
\ref{item:jjtop}
is deduced from 
Statement \ref{item:osmisatop} in 
Proposition 
\ref{prop:ellisom}, 
Corollary 
\ref{cor:wassametop},
and 
 Proposition 
 \ref{prop:cbasicstop}. 
\end{proof}

Recall that 
if 
$X$ 
is completely metrizable, 
then 
we choose 
$m$ 
and 
$w$
as  complete metrics on 
$A$ and 
$X$, 
respectively. 

\begin{prop}\label{prop:jxclosed}
If 
$X$
 is 
 a
 completely metrizable 
 space, 
then 
for every complete 
metric 
$d$ 
in 
$\yoccl$, 
the metric  subspace  
$(\yosubmap(X), \yoosmislldwdis{d})$
of 
$\yoosmislldwsp{A}$
 is complete.  
In particular, 
the map 
$\yosubmap\colon X\to \yoosmislldwsp{A}$
is a closed map. 
\end{prop}
%
\begin{proof}
Assume that 
a sequence 
$\{\yosubmap(x_{i})\}_{i\in \zz_{\ge 0}}$ 
is Cauchy with respect to 
$\yoosmislldwdis{d}$. 
We will show that 
$\{x_{i}\}_{i\in \zz_{\ge 0}}$
has a 
subsequence 
possessing  a 
limit 
$p\in X$
with respect to 
the topology of
 $X$. 
If
 $\{x_{i}\}_{i\in \zz_{\ge 0}}$
has a subsequence contained in 
$A$, 
then 
we can find 
$p\in X$
such that 
$x_{i}\to p$
as 
$i\to \infty$
since 
$(A, d)$
 is complete
 and 
 $\yosubmap$ 
 is isometric on 
 $(A, d)$. 
Thus, 
we may assume that 
$x_{i}\in X\setminus A$
for every 
$i\in \zz_{\ge 0}$
by taking 
a
subsequence if necessary. 
We divide 
the proof into two cases.

\yocase{1}{$\inf_{i\in \zz_{\ge 0}}w(x_{i}, A)>0$}
In this case, 
we can find 
$n\in \zz_{\ge 0}$
 such that 
$2^{-n+1}\le \inf_{i\in \zz_{\ge 0}}w(x_{i}, A)$. 
Then there exists 
$K\in \zz_{\ge 0}$
 such that 
for every 
$i\in \zz_{\ge 0}$
 with 
$K\le i$,
 we have 
$\yofmodisn(\yohsmapr(x_{i}, n), \yohsmapr(x_{K}, n))<1$. 
Then 
Statement 
\ref{item:hrle12}
 in 
Lemma 
\ref{lem:hrmap}
shows that 
$w(x_{i}, x_{K})\le 16\cdot 4^{-n-1}w(x_{K}, A)$. 
In this case, 
we have 
\begin{align*}
w(x_{i}, A)&\le w(x_{K}, A)+w(x_{i}, x_{K})
\le (1+16\cdot 4^{-n-1})w(x_{K}, A)
\\
&\le 17w(x_{K}, A).
\end{align*} 
Thus, 
we obtain 
\[
\sup_{i\in \zz_{\ge 0}}w(x_{i}, A)<\infty.
\] 
Put 
$C=\sup_{i\in \zz_{\ge 0}}w(x_{i}, A)$. 
Then
for every 
$\epsilon \in (0, \infty)$,  
we can find 
$N\in \yoomega$
such that 
$16\cdot 4^{-N-1}C\le \epsilon$. 
In this setting, 
for sufficiently large numbers 
$i, j\in \zz_{\ge 0}$, 
 we have 
$\yocoefff{N}(x_{i})=\yocoefff{N}(x_{j})=1$
and 
\[
\yofmodisn(\yohsmapr(x_{i}, N), \yohsmapr(x_{j}, N))<1.
\] 
Thus 
Statement 
\ref{item:hrle12} in 
Lemma 
\ref{lem:hrmap}
shows that 
\[
w(x_{i}, x_{j})\le 16\cdot 4^{-N-1}w(x_{i}, A)\le
16\cdot 4^{-N-1} C\le \epsilon.
\] 
Therefore 
the sequence  
$\{x_{i}\}_{i\in \zz_{\ge 0}}$
 is 
Cauchy in 
$(X, w)$. 
Hence, 
using the completeness of 
$w$, 
we can find a limit 
$p\in X$  
of 
$\{x_{i}\}_{i\in \zz_{\ge 0}}$.

\yocase{2}{$\inf_{i\in \zz_{\ge 0}}w(x_{i}, A)=0$}
We may assume that 
$w(x_{i}, A)\to 0$
as 
$i\to \infty$
by taking a subsequence if necessary. 
Since 
\[
\yowasdis{m}(\yohsmapl(x_{i}, 0), \yohsmapl(x_{i}, 0))\le 
\yoosmislldwdis{m}(\yosubmap(x_{i}), \yosubmap(x_{j})), 
\]
the sequence 
$\{\yohsmapl(x_{i}, 0)\}_{i\in \zz_{\ge 0}}$
 is 
Cauchy in 
$(\yowassp{A}, \yowasdis{d})$. 
Thus we can find 
$\theta \in \yowassp{A}$
 such that 
$\yohsmapl(x_{i}, 0)\to \theta$
as 
$i\to \infty$
since 
the space 
$(\yowassp{A}, \yowasdis{d})$
is complete 
(see 
\ref{item:wascomp}
in 
Proposition 
\ref{prop:wasisom}). 
This implies that 
the sequence 
$\{\yohsmapl(x_{i}, 0)\}_{i\in \zz_{\ge 0}}$ is 
also
Cauchy in 
$(\yowassp{A}, \yowasdis{m})$. 
For each 
$i\in \zz_{\ge 0}$, 
take 
$b_{i}\in A$ 
such that 
$w(x_{i}, b_{i})\le 2w(x_{i}, A)$. 
Then 
Lemma 
\ref{lem:wasfin1}
shows that 
\[
\yowasdis{m}(\yodmas{b_{i}}, \yohsmapl(x_{i}, 0))\le 
w(x_{i}, b_{i})\le 2w(x_{i}, A).
\] 
Thus, for sufficiently large 
numbers 
$i, j\in \zz_{\ge 0}$, 
using
Corollary 
\ref{cor:oswasext}, 
we have 
\begin{align*}
&w(b_{i}, b_{j})
=m(b_{i}, b_{j})
=
\yowasdis{m}(\yodmas{b_{i}}, \yodmas{b_{j}})\le \\
&
\yowasdis{m}(\yodmas{b_{i}}, \yohsmapl(x_{i}, 0))+
\yowasdis{m}(\yohsmapl(x_{i}, 0), \yohsmapl(x_{j}, 0))+\\
&\yowasdis{m}(\yohsmapl(x_{j}, 0), \yodmas{b_{j}})\\
&\le 
2w(x_{i}, A)+\yowasdis{m}(\yohsmapl(x_{i}, 0), \yohsmapl(x_{j}, 0))+2w(x_{j}, A)\to 
0, 
\end{align*}
as 
$i, j\to \infty$. 
This means that 
the sequence 
$\{b_{i}\}_{i\in \zz_{\ge 0}}$
 is 
Cauchy in 
$(A, m)$. 
Hence there exists 
$p\in A$ 
such that 
$b_{i}\to p$
as 
$i\to \infty$. 
Since 
$w(b_{i}, x_{i})\le 2w(x_{i}, A) \to 0$
 as
$i\to \infty$, 
we also see that 
$x_{i}\to p$ 
as 
$i\to \infty$.

In any case, 
we can find a subsequence of 
$\{x_{i}\}_{i\in \zz_{\ge 0}}$
possessig a limit  
$p\in X$. 
Due to the continuity of 
$\yosubmap$ 
(Lemma \ref{lem:jconti}), 
the Cauchy sequence 
$\{\yosubmap(x_{i})\}_{i\in \zz_{\ge 0}}$
has a limit 
$\yosubmap(p)$. 
Since 
$\yosubmap$ is a
topological embedding
and since 
the image 
$\yosubmap(X)$
is a complete subspace of  
$\yoosmislldwsp{A}$, 
we see that
$\yosubmap$ is 
a closed map. 
This finishes the proof. 
\end{proof}

\begin{rmk}
Some readers may 
feel that 
the proof of 
Proposition 
\ref{prop:jxclosed}
is strange 
because 
it does not use the 
essential 
information of 
$\yoosmislldwdis{d}$, 
or  because 
the whole space 
$\yoosmislldwsp{A}$ is not 
seem to be 
complete since 
neither is  
$S$. 
However, 
the embedding 
$\yosubmap$
is constructed using 
$w$
and 
the  
sequence of 
$(k, w)$-WD 
collections, 
which contains all information of 
the topology  of 
$X\setminus A$. 
Indeed, 
the family 
$\bigcup_{s\in \yoomega}\yowdce{s}$
becomes an open basis on 
$X\setminus A$. 
Therefore, 
we can draw out the 
information of
the topology and  the 
completeness 
from 
the embedding 
$\yosubmap$. 
\end{rmk}

To construct 
the map 
$\yomainmap$
stated in Theorem
 \ref{thm:main1}, 
 it should be 
noted that 
we can represent  
$\met{A}=\coprod_{\yoccl\in \yoslamet{X}}\yoccl$. 
For each 
$\yoccl\in \yoslamet{X}$, 
choose 
$m_{\yoccl}\in \yoccl$
and 
$w_{\yoccl}\in \met{X}$
such that 
$w_{\yoccl}|_{A^{2}}=m_{\yoccl}$. 
Using these
$m_{\yoccl}$ 
and 
$w_{\yoccl}$, 
and using 
the above  discussion, 
we obtain an isometric operator 
$\yosubmap^{\yoccl}\colon \yoccl\to \met{X}$. 
Then 
we define a map 
$\yomainmap\colon \met{A}\to \met{X}$
by 
\[
\yomainmap(d)
(x, y)=\yoosmislldwdis{d}
\left(\yosubmap^{\yoccl}(x), \yosubmap^{\yoccl}(y)\right). 
\]
Proposition 
\ref{prop:jhomeo}
shows that 
$\yomainmap(d)$
actually belongs to 
$\met{X}$
for all 
$d\in \yoccl$.

Now, 
let us complete the proof of 
Theorem \ref{thm:main1}. 
\begin{proof}[Proof of Theorem \ref{thm:main1}]
First, Proposition 
\ref{prop:jextension}
and 
Proposition 
\ref{prop:jxclosed}
 proves 
Statements
 \ref{item:m1ext} 
 and 
 \ref{item:m1comp}, 
 respectively. 
Statement 
\ref{item:m1bounded}
follows from 
Lemma 
\ref{lem:jbounded}.  
Now we shall show 
Statement  
\ref{item:m1isom}. 
Take
 $d, e\in \met{A}$.
If 
$\yocclq{d}\neq \yocclq{d}$,
then 
$\metdis_{A}(d, e)=\infty$. 
Hence  
we have 
$\metdis_{X}(\yomainmap(d), \yomainmap(e))=\infty
=\metdis_{A}(d, e)$. 
Thus, 
we only need to show that 
for a fixed class 
$\yoccl\in \yoslamet{A}$, 
and 
for every pair 
$d, e\in \yoccl$, 
we have 
\[
\metdis_{A}(d, e)
=\metdis_{X}(\yomainmap(d), \yomainmap(e)).
\]
Theorem 
\ref{thm:osmotic} 
indicates that 
\begin{align*}
&\metdis_{A}(d, e)\le 
\metdis_{X}(\yomainmap(d), \yomainmap(e))
=\metdis_{\yosubmap(X)}(\yoosmislldwdis{d}, 
\yoosmislldwdis{e})
\\
&\le 
\metdis_{\yoosmislldwsp{A}}(\yoosmislldwdis{d}, 
\yoosmislldwdis{e})
\\
&=
\metdis_{\yoosmisdwsp{A}}(\yoosmisdwdis{d}, 
\yoosmisdwdis{e})
=\metdis_{\yoosmisdwsp{A}}(\yoosmisdwdis{d}, 
\yoosmisdwdis{e})\\
&=
\metdis_{\yowassp{A}}(\yowasdis{d}, \yowasdis{e})
=
\metdis_{A}(d, e). 
\end{align*}
Therefore 
we conclude that 
$\metdis_{A}(d, e)
=\metdis_{X}(\yomainmap(d), \yomainmap(e))$.
This verifies Statement  
\ref{item:m1isom}.

Next we consider the 
``furthermore'' part of the theorem. 
Since 
$\met{X}$
is dense 
in 
$\yocpm{X}$ 
(\cite[Theorem 1.3]{Ishiki2024smbaire}), 
the isometric embedding 
$\yomainmap\colon \met{A}\to \met{X}$
can be (uniquely) 
extended to an isometric embedding 
$\yomainmapf\colon \yocpm{A}\to \yocpm{X}$
using Cauchy sequences
(see also 
\cite[Theorem 2]{MR0390999} and \cite{MR0969516}). 
Let us  show the equality 
\[
\yomainmapf(\met{A})=\yomainmapf(\yocpm{A})\cap \met{X}.
\] 
By the definitions of $\met{X}$, $\met{A}$ and $\yocpm{A}$, 
and the fact that 
$\yomainmapf|_{\met{A}}=\yomainmap$, 
the inclusion 
$\yomainmapf(\met{A})\yosub \yomainmapf(\yocpm{A})\cap \met{X}$ is true. 
Take $d\in \yomainmapf(\yocpm{A})\cap \met{X}$. 
Then  there exists $\psi\in \yocpm{A}$ such that 
$d=\yomainmapf(\psi)$. 
Since 
$\yomainmapf(\psi)|_{A\times A}=\psi$, 
and $d\in \met{X}$, 
we see that 
$\psi\in \met{A}$.
Thus $d\in \yomainmap(\met{A})$. 
This implies that 
$\yomainmapf(\met{A})=\yomainmapf(\yocpm{A})\cap \met{X}$. 
Since 
$\yocpm{A}$ is complete and 
$\yomainmapf$ is isometric, 
the subspace 
$\yomainmapf(\yocpm{A})$ is 
closed in 
$\yocpm{X}$. 
Namely, 
$\yomainmap(\met{A})$ 
is closed in 
$\met{X}$. 
This finishes 
the whole of the 
proof of Theorem 
\ref{thm:main1}. 
\end{proof}



\section{Questions}\label{sec:ques}

In the paper 
\cite{koshino2024borel}, 
as a consequence of research on Borel hierarchy 
of $\met{X}$, 
Koshino proved that, 
when $X$ is separable, 
$\met{X}$ is completely metrizable if and only if 
$X$ is $\sigma$-compact (\cite[Corollary]{koshino2024borel}). 
Motivated by the 
aim 
 to remove the assumption of the  separability of 
 $X$ in  
 this theorem, we propose the following conjecture. 
\begin{conj}\label{conj:unctb}
Recall that 
$\aleph_{1}$
stands for  the first uncountable 
cardinality, and 
let 
$D_{\aleph_{1}}$ denote the
the discrete space of cardinality of 
$\aleph_{1}$. 
Under this notations, 
the space
$\met{D_{\aleph_{1}}}$ is not 
completely metrizable. 
\end{conj}
Take a non-separable metrizable space 
$X$. 
Then 
$X$ 
contains 
$D_{\aleph_{1}}$ 
as a closed subset. 
Thus, 
$\met{X}$ contains 
$\met{D_{\aleph_{1}}}$ as a closed subset 
due to 
Theorem \ref{thm:main1}. 
If Conjecture 
\ref{conj:unctb}
is true, then  
the space 
$\met{X}$
 would 
not be completely metrizable. 
Namely, 
the complete metrizability of 
$\met{X}$ 
would imply
the separability of
 $X$. 
This observation is a reason why 
the author supports Conjecture 
\ref{conj:unctb}. 

\begin{ques}
Let 
$X$
 be a 
 metrizable space, 
 and 
 $A$
 be a closed subset of 
 $X$.
 Can we obtain 
 a linear 
 isometric operator 
 extending metrics?
 Namely, 
does there exists an
extensor 
$G\colon \met{A}\to\met{X}$
satisfying the conclusions in Theorem 
\ref{thm:main1} and the following 
additional condition? 
\begin{itemize}
\item 
For every pair $d, e\in \met{A}$, and 
$s, t\in (0, \infty)$, 
we have 
\[
G(s\cdot d +t\cdot e)=
s\cdot G(d)+t\cdot G(e). 
\]
\end{itemize}
The author thinks that 
this questions is true 
if 
we can prove 
that every pair 
$d, e\in \met{A}$ satisfies that 
$\yowasdis{d+e}=\yowasdis{d}+\yowasdis{e}$; 
however, it seems to be  quite difficult. 
Thus, we need another osmotic construction 
preserving linear combinations  of metrics instead of 
$\yowas$. 
\end{ques}

In a non-Archimedean case, 
we can obtain an 
isometric extensor 
preserving fractal dimensions
(see \cite[Theorem 4.7]{Ishiki2022factor})
such as 
the Hausdorff dimension, the packing dimension, the upper box dimension, and the Assouad dimension. 
We ask whether we can obtain Archimedean 
analogue of this result.

\begin{ques}
Let 
$X$
 be a 
 metrizable space, 
 and 
 $A$
 be a closed subset of 
 $X$.
 Assume that 
 $A$ 
 and
  $X$ 
  have 
 the same topological dimensions. 
For each 
$d\in \met{A}$, 
does there exist a metric 
$D\in \met{X}$
such that 
the fractal dimensions 
of 
$d$ 
and
 $D$ 
 are identical? 
Moreover, 
can we obtain an 
isometric extensor 
of metrics preserving 
fractal dimensions?
\end{ques}

A metric on a set 
$Z$ 
is said to be 
\yoemph{proper}
if 
every 
 bounded set in 
$(Z, d)$
is compact. 
For a metrizable space 
$X$, 
we denote  by 
$\yoprmet{X}$
the set of all 
$d\in \met{X}$
that is proper.
In the 
paper 
\cite{MR4527953}, 
the author 
obtained 
an analogue of 
Hausdorff's metric extension theorem 
for proper metrics. 
It is interesting 
to ask 
whether we 
construct a 
simultaneous extension of 
proper metrics or not.

\begin{ques}
Let 
$X$
 be a 
 second-countable 
 locally compact Hausdorff space, 
 and 
 $A$
 be a closed subset of 
 $X$.
Does there exist
an extensor 
$F\colon \met{A}\to \met{X}$
satisfying the conclusions of 
Theorem \ref{thm:main1}
and the  additional 
condition that 
$F(\yoprmet{A})\yosub \yoprmet{X}$?
\end{ques}

\begin{ques}
Similarly to 
$\met{X}$, 
can we investigate 
the topology and comeager subsets 
of 
the space of proper metrics 
equipped with the supremum metric?
\end{ques}



\bibliographystyle{myplaindoidoi}
\bibliography{../../bibtexmet/bibmet.bib}



\end{document}